\documentclass{article}%

\usepackage{amsmath,amsthm,amsfonts,amssymb,mathtools,hyperref,tikz-cd,stackrel,enumerate}
\usepackage{stmaryrd} 
\usepackage{mathabx} 
\usepackage[a4paper, total={6in, 9in}]{geometry} 
\usepackage[sc, small, center]{titlesec} 
\tikzset{%
    symbol/.style={%
        draw=none,
        every to/.append style={%
            edge node={node [sloped, allow upside down, auto=false]{$#1$}}}
    }
}

\theoremstyle{plain} 
\newtheorem{theorem}{Theorem}[section] 
\newtheorem{lemma}[theorem]{Lemma} 
\newtheorem{proposition}[theorem]{Proposition}
\newtheorem{corollary}[theorem]{Corollary}

\theoremstyle{definition} 
\newtheorem{definition}[theorem]{Definition}
\newtheorem{example}[theorem]{Example}
\newtheorem{notation}[theorem]{Notation}
\newtheorem{remark}[theorem]{Remark}

\DeclareMathOperator \Orbop {\mathrm{Orb}_\G^{\ab,\op}}
\DeclareMathOperator \Map {Map}
\DeclareMathOperator \Lie {Lie}
\DeclareMathOperator \Hom {Hom}

\DeclareMathOperator \Aut {Aut}

\DeclareMathOperator \Ran {Ran}
\DeclareMathOperator \Diff {Diff}
\DeclareMathOperator \CxMfd {CxMfd}
\DeclareMathOperator \Open {Open}

\DeclareMathOperator \Sh {Sh}
\DeclareMathOperator \Ring {Ring}
\DeclareMathOperator \Fun {Fun}
\DeclareMathOperator \Ell {\mathcal{E\ell\ell}}

\DeclareMathOperator \Orb {Orb^{ab}_{\G}}

\DeclareMathOperator \Nat {Nat}

\DeclareMathOperator \Gau {Gau}
\DeclareMathOperator \SL {SL}
\DeclareMathOperator \GL {GL}

\DeclareMathOperator \ind {ind}
\DeclareMathOperator \res {res}
\DeclareMathOperator \op {op}

\DeclareMathOperator \Res {Res}

\DeclareMathOperator \Top {Top}

\newcommand{\Topv}[1]{#1\text{-}\mathrm{Top}}
\newcommand{\Fin}[1]{#1\text{-}\mathrm{Fin}}

\definecolor{blueish}{HTML}{196F9E}

\def \1 {\mathbf{1}}
\def \H {\mathcal{H}}
\def \O {\mathcal{O}}
\def \X {\mathcal{X}}
\def \G {\mathcal{G}}

\def \F {\mathcal{F}}

\def \M {\mathcal{M}}
\def \E {\mathcal{E}}

\def \J {\mathcal{J}}
\def \W {\mathrm{W}}

\def \Z {\mathbb{Z}}
\def \C {\mathbb{C}}

\def \R {\mathbb{R}}
\def \bT {\mathbb{T}}

\def \L {\mathcal{L}}
\def \Fr {\mathrm{Fr}}
\def \pt {\mathrm{pt}}

\def \id {\mathrm{id}}
\def \ev {\mathrm{ev}}
\def \ab {\mathrm{ab}}

\def \to {\rightarrow}

\begin{document}

\begin{center} 
{\bf EQUIVARIANT ELLIPTIC COHOMOLOGY AND THE 2-LOOP GROUPOID} \\
\vskip 8mm

M. SPONG

\end{center}

\vskip 8mm

\begin{abstract}
Following an outline of Rezk, we give a construction of complex-analytic $G$-equivariant elliptic cohomology for an arbitrary compact Lie group $G$ and we prove some of its fundamental properties. The construction is parametrised over the orbit category of the groupoid of principal $G$-bundles over 2-dimensional tori and generalises Grojnowski's construction of equivariant elliptic cohomology.
\end{abstract} 

\tableofcontents

\footnotetext{2020 Mathematics Subject Classification 55N34}

\section{Introduction}

Elliptic cohomology is a type of cohomology theory associated with elliptic curves that first appeared in the late 1980s in \cite{LRS}. An important motivation for its construction was the elliptic genus, which was shown by Witten to be related to certain invariants of 2-dimensional quantum field theories in \cite{Witten}. It has since been an open problem to find a geometric interpretation of elliptic cohomology in terms of 2-dimensional field theories that is analogous to the interpretation of K-theory in terms of vector bundles.

The problem of finding a geometric meaning becomes more tractable in the equivariant context. The first equivariant version of elliptic cohomology was built in the early 1990s by Grojnowski \cite{Groj} in the complex-analytic setting. Although this construction was highly useful with applications in both geometric representation theory (e.g. \cite{GKV}) and homotopy theory (e.g. \cite{Rosu1}), in its original form it did not shed much light on the geometry underlying elliptic cohomology. Recently, there have appeared several constructions of complex-analytic equivariant elliptic cohomology in terms of 2-dimensional quantum field theories, in particular work of Kitchloo in \cite{Kitchloo} and work of Berwick-Evans and Tripathy in \cite{BET1} and \cite{BET2}, the latter continuing the program of Segal \cite{Segal} and of Stolz and Teichner \cite{StolzTeichner}. On the other hand, in \cite{GepnerMeier} Gepner and Meier have given a homotopy-theoretic construction of equivariant elliptic cohomology over the integers along lines sketched by Lurie in \cite{Lurie}. A clearer understanding of the link between constructions originating in physics with those originating in homotopy theory would contribute towards the geometric interpretation of elliptic cohomology over the integers. 

The construction in this paper establishes a concrete and functorial relationship between the topological category of flat $G$-bundles on genus 1 surfaces on the one hand and $G$-equivariant elliptic cohomology on the other. It follows the outline given by Rezk in \cite{Rezktalk} (see also \cite{Rezk}) which incorporates work of the author in \cite{Spong}, and may be viewed as a modern version of Grojnowski's complex-analytic equivariant theory, situated between constructions based in physics and constructions coming from homotopy theory. We summarise the construction as follows.  \\

\noindent {\bf Summary of the construction.} Let $G$ be a compact Lie group.  The first ingredient in the construction is the $2$-loop groupoid $\L^2_G$ of $G$. An object of $\L^2_G$ is a smooth principal $G$-bundle $P$ over a 2-dimensional torus $\Sigma$ and a morphism is a $G$-diffeomorphism $P \to P'$ covering a diffeomorphism $\Sigma \to \Sigma'$ of base manifolds. We regard $\L^2_G$ as a topologically enriched groupoid by providing each set of maps with the compact-open topology. 

Given an object $P$ in $\L^2_G$, we make use of Rezk's notion of a $\C$-frame in $P$, which is a continuous homomorphism $\gamma: \C \to \Aut(P)$ that induces a transitive $\C$-action on the base manifold $\Sigma$ of $P$. We show that there is a canonical bijection
\begin{align*}
\Fr(P) &\xrightarrow{\sim} \{ \text{complex structures on $\Sigma$ + flat connection on $P$} \} \\
\gamma &\mapsto (\Lambda_\gamma, \eta_\gamma,\nu_\gamma)
\end{align*}
between the set of all frames in $P$ and the set of triples $(\Lambda,\eta,\nu)$, where $\C/\Lambda$ is a complex elliptic curve, $\eta: \C/\Lambda \cong \Sigma$ is a diffeomorphism (unique up to translation in $\C/\Lambda$) and $\nu$ is a flat connection on $P\to \Sigma$. It turns out that for each frame $\gamma$ in $P$ there exists a torus $T \leq \Aut(P)$ containing the image of $\gamma$. Thus, by respectively considering the subsets 
\[
\Fr(T):= \{\gamma: \C \to \Aut(P) \ | \ \gamma(\C) \subset T \} \ \subset \Fr(P)
\]
for each torus $T$ in $\Aut(P)$, we capture all frames in $P$. This motivates Rezk's definition of the abelian orbit category $\mathrm{Orb}^{\ab}_{\L^2_G}$ of the topological groupoid $\L^2_G$. This category may be seen as a generalisation of the abelian orbit category of a topological group, where one considers only objects with compact abelian isotropy. An object is a pair $(P,T)$, where $P$ is an object in $\L^2_G$ and $T \leq \Aut(P)$ is a compact abelian subgroup, and the automorphism group $W(\Aut(P)/T)$ of $(P,T)$ may be identified with the Weyl group of $T$ in $\Aut(P)$.

The set $\Fr(T)$ is a $\C^\times$-equivariant open subset of the complex vector space $\Hom(\C,T)$ of Lie group homomorphisms. We show that there is a functor
\begin{align*}
\Fr: \mathrm{Orb}^{\ab}_{\L^2_G} &\longrightarrow \CxMfd^{\C^\times}  \\
(P,T) &\longmapsto \Fr(T)
\end{align*}
into the category of $\C^\times$-equivariant complex manifolds with $\C^\times$-equivariant holomorphic maps. In particular, $\Fr(T)$ has the structure of a $\C^\times \times W(\Aut(P)/T)$-equivariant complex manifold and the action corresponds to a suitable notion of equivalence of the geometric data underlying different frames in $T$. 

Another ingredient in the construction is a certain category $\Sh_{\Fr}^{\C^\times}$ of $\C^\times$-equivariant holomorphic sheaves over spaces of the form $\Fr(T)$. This category comes equipped with a functor 
\[
\Sh_{\Fr}^{\C^\times}\longrightarrow \mathrm{Orb}^{\ab,\op}_{\L^2_G}
\]
which is a cocartesian fibration. 

A further ingredient is complex analytic $T$-equivariant cohomology, which is a functor $\H^\bullet_T$ defined on $T$-equivariant spaces and taking values in $\C^\times$-equivariant, $\Z$-graded holomorphic sheaves over $\Hom(\C,T)$. We delay the actual definition of $\H^\bullet_T$ until later in the introduction. 

The main construction of the paper is a functor
\[
\Ell^\bullet_G: \Fin{G}^{\op} \longrightarrow \Gamma( \mathrm{Orb}^{\ab,\op}_{\L^2_G}, \Sh_\Fr^{\C^\times})
\]
defined on the category of finite $G$-CW complexes and taking values in sections of the above fibration. The functor sends $X$ to a section whose value at $(P,T)$ is the sheaf  
\[
\Ell^\bullet_{G,P,T}(X) := \H^\bullet_T(\Map_G(P,X))_{\Fr(T)}
\]
which we call the $G$-equivariant elliptic cohomology of $X$ at $(P,T)$. Here $\Map_G(P,X)$ is the $T$-space of $G$-equivariant maps $P \to X$, with $T$ acting by precomposition, and the subscript $\Fr(T)$ denotes the restriction of the sheaf along $\Fr(T) \hookrightarrow \Hom(\C,T)$. 

Thus, $\Ell^\bullet_{G,P,T}(X)$ is a $\C^\times \times W(\Aut(P)/T)^{\op}$-equivariant holomorphic sheaf over $\Fr(T)$ with structure map
\[
\O^\bullet_{\Fr(T)} = \Ell^\bullet_{G,P,T}(\pt) \to \Ell^\bullet_{G,P,T}(X)
\]
induced by the map from $X$ to a point. Finally, it is important to note that $\Ell^\bullet_{G,P,T}(X)$ is trivial for all bundles $P$ that do not admit a flat connection, since in that case $\Fr(T)$ is empty for all $T \leq \Aut(P)$.   \\

\noindent {\bf Main results.} Our main results rely on the following localisation theorem for $\H^\bullet_T$, which is Theorem \ref{local} in the body of the paper and is adapted from Theorem 7.2 in \cite{Spong}.\\

\noindent LOCALISATION THEOREM.
{\em Let $S$ be a $\C^\times$-equivariant open subset of $\Hom(\C,T)$ and, for a $T$-space $Y$, write $Y^S$ for the union
\[
Y^S := \bigcup_{\gamma \in S} Y^\gamma
\]
where $Y^\gamma$ denotes the fixed-point subspace $Y^{\gamma(\C)}$ of $Y$. The map induced by the inclusion $Y^S \hookrightarrow Y$ of $T$-spaces is an isomorphism
\[
\H^\bullet_T(Y)_S \cong \H^\bullet_T(Y^S)_S
\]
of $\C^\times$-equivariant sheaves over $S$.}\\

The localisation theorem yields interesting results if we set $S = \Fr(T)$ and $Y = \Map_G(P,X)$ for the data $(X,P,T)$ as defined above. In particular, it enables us to show that the reduced version of $\Ell^\bullet_{G}$ is exact and has a canonical suspension isomorphism. Combining the results of Propositions \ref{opcarry}, \ref{exactness} and \ref{susp}, we have the following theorem. \\

\noindent THEOREM A.
{\em The functor $\Ell^\bullet_G$ is a cohomology theory and takes values in cocartesian sections.}\\

The most useful objects $(P,T)$ at which to compute $\Ell^\bullet_{G,P,T}(X)$ are those which are maximal in the sense that the identity component of $T$ is maximal among tori in $\Aut(P)$. It is precisely these objects that have discrete automorphism groups $W(\Aut(P)/T)$. 

Following somewhat more directly from the localisation theorem for $\H^\bullet_T$ are Propositions \ref{locallydokelly} and \ref{walkystalky}, which we combine in the following theorem.  \\

\noindent THEOREM B.\hfill
\begin{enumerate}[(i)]
\item {\em  The inclusion $\Map_G(P,X)^{\Fr(T)} \hookrightarrow \Map_G(P,X)$ of $T$-spaces induces an isomorphism
\[
\Ell^\bullet_{G,P,T}(X) \xrightarrow{\sim} \H^\bullet_T(\Map_G(P,X)^{\Fr(T)})_{\Fr(T)}
\]
in $\Sh_\Fr^{\C^\times}$ natural in $(X,P,T)$.}
\item {\em  For any $\gamma \in \Fr(T)$, the inclusion $\Map_G(P,X)^{\gamma} \hookrightarrow \Map_G(P,X)$ of $T$-spaces induces an isomorphism
\[
\Ell^\bullet_{G,P,T}(X)_\gamma  \xrightarrow{\sim} \H^\bullet_T(\Map_G(P,X)^{\gamma})_\gamma
\]
of stalks at $\gamma$.} \end{enumerate} 

The fixed-point subspace $\Map_G(P,X)^{\gamma}$ which appears in the theorem above may be identified with the space of sections of the associated bundle $P \times_G X \to \Sigma$ which are parallel with respect to the flat connection encoded by $\gamma$. \\

\noindent {\bf Relationship to elliptic curves.} We justify the claim that $\Ell^\bullet_G$ ought to be regarded as a version of $G$-equivariant elliptic cohomology. The subsheaf of invariants
\[
\Ell^\bullet_{G,P,T}(X)^{\C^\times \times W(\Aut(P)/T)^{\op}}
\]
is equivalently a sheaf on the stack
\begin{equation}\label{stack}
[\Fr(T) \sslash \C^\times \times W(\Aut(P)/T)]
\end{equation}
which classifies frames in $T$. Let $G = U(1)$, let $P$ be the trivial bundle $U(1) \times \bT^2$ over the standard torus $\bT^2 = \R^2/\Z^2$ and let $T$ be the maximal torus $U(1) \times \bT^2$ in
\[
\Aut(P) \cong \Map(\bT^2,U(1)) \rtimes \Diff(\bT^2)
\]
consisting of constant maps $\bT^2 \to U(1)$ and translations of $\bT^2$. Then $W(\Aut(P)/T)$ is what is sometimes called the {\em elliptic Weyl group} of $U(1)$
\[
W(\Aut(P)/T) \cong \Hom(\bT^2,U(1)) \rtimes \GL_2(\Z)
\]
and the stack \eqref{stack} is isomorphic to the universal elliptic curve $\E$ over the moduli stack $\M_{ell}(\C)$ of complex elliptic curves. In fact, with these hypotheses, 
\[
\Ell^\bullet_{G,P,T}(X)^{\C^\times \times W(\Aut(P)/T)^{\op}}
\]
recovers Grojnowski's seminal construction of the $U(1)$-equivariant elliptic cohomology \cite{Groj} of $X$ (this follows from the main result in \cite{Spong}, reformulated in this paper in Proposition \ref{groj}). Moreover, if $G$ is the trivial group $e$, $P = e \times \bT^2 \to \bT^2$, $T$ is the group of translations of $\bT^2$ and $X$ is a point, then the global sections of the invariant subsheaf
\begin{align*}
\Ell^\bullet_{e,P,T}(\pt)^{\C^\times \times W(\Aut(P)/T)^{\op}} \cong (\O^\bullet_{\Fr(T)})^{\C^\times \times \GL_2(\Z)^{\op}}
\end{align*}
is the graded ring $MF^\bullet_{weak}$ of weak modular forms (this also follows from Remark 7.12 in \cite{Spong} and the discussion preceding Proposition \ref{groj}). \\

\noindent {\bf Complex-analytic $T$-equivariant cohomology.} Let $T^\C$ denote the complex vector space $\Hom(\C,T)$. Complex-analytic $T$-equivariant cohomology is a contravariant functor\footnote{The functor $\H^\bullet_T$ adapts Definition 6.2 in \cite{Spong} to the context of this paper.}
\[
\H^\bullet_T: \Topv{T}^{\op} \to \Sh(T^\C,\O^\bullet)^{\C^\times}
\]
from the category of $T$-spaces into the category of certain $\C^\times$-equivariant holomorphic sheaves over $T^\C$. We construct it in three steps:
\begin{enumerate}
\item We define a $\C^\times$-equivariant sheaf $\O^\bullet_T$ of graded rings over $T^\C$. It is concentrated in even degrees and its $2k$th piece is given by the sheaf of holomorphic functions $U \to \omega^{\otimes k}$, where $U \subset T^\C$ is an open subset and $\omega := \Hom(\C,\R/\Z)$. 
\item We define a functor on finite $T$-CW complexes
\[
\F^\bullet_T: \Fin{T}^{\op} \longrightarrow \Sh(T^\C,\O^\bullet)^{\C^\times}
\]
as follows. For a finite $T$-complex $Y$, let $\underline{H}_T^\bullet(Y)$ be the constant sheaf associated to the Borel $T$-equivariant cohomology $H^\bullet(Y \times_T ET;\Z)$ of $Y$. We regard this sheaf as equipped with the trivial $\C^\times$-action. There are canonical maps
\[
\underline{H}^\bullet_T(\pt) \to \underline{H}_T^\bullet(Y) \quad \text{and} \quad \underline{H}^\bullet_T(\pt) \to \O^\bullet_T
\]
of $\C^\times$-equivariant sheaves of graded rings, where the first is induced by $Y \to \pt$ and the second regards an element of $H^\bullet(BT;\Z)$ as a holomorphic function on $T^\C$. The functor $\F^\bullet_T$ sends $Y$ to the tensor product
\[
\F^\bullet_T(Y) := \underline{H}_T^\bullet(Y) \otimes_{\underline{H}^\bullet_T(\pt)} \O^\bullet_{T}
\]
of $\C^\times$-equivariant $\underline{H}^\bullet_T(\pt)$-algebras.
\item We define $\H^\bullet_T$ as the right Kan extension of $\F^\bullet_T$ along the inclusion $\Fin{T}^{\op} \hookrightarrow \Topv{T}^{\op}$ of finite $T$-CW complexes into $T$-spaces.
\end{enumerate}

The functor $\H_T^\bullet$ is not actually a cohomology theory on all $T$-spaces because it is not exact on non-finite spaces. For any given $\C^\times$-equivariant open subset $S$ of $T^\C$, we show that there is a subcategory 
\[
\Topv{T}(S) \subset \Topv{T}
\]
such that the restriction
\[
Y \mapsto \H^\bullet_T(Y)_S
\]
of the sheaf to $S$ is exact on $\Topv{T}(S)$. This subcategory is defined as the full subcategory of spaces $Y$ such that for each $\gamma \in S$, there is a $\C^\times$-equivariant open neighbourhood $\gamma \in U_\gamma \subset S$ satisfying:
\begin{enumerate}[(i)]
\item $Y^{U_\gamma}$ has the $T$-homotopy type of a finite $T$-CW complex, and
\item $Y^{U_\gamma} = Y^{\gamma}$.
\end{enumerate}
If these conditions are satisfied, then we say that $Y$ is {\em locally finite over $S$}. 
Such spaces satisfy a stronger version of the localisation theorem (Proposition \ref{stalks}).\\

\noindent STRONG LOCALISATION THEOREM. {\em Let $Y$ be a $T$-space that is locally finite over $S$ and let $\gamma \in S$. The inclusion map $Y^{\gamma} \hookrightarrow Y$ induces an isomorphism
\[
\H^\bullet_T(Y)_\gamma \cong \H^\bullet_T(Y^{\gamma})_\gamma 
\]
of stalks at $\gamma$.}\\

We show in Section \ref{locally} that, for a finite $G$-complex $X$, the $T$-space $\Map_G(P,X)$ is locally finite over $\Fr(T)$ and that each map in $\Map_G(P,X)^\gamma$ has image contained in a single $G$-orbit of $X$. These results allow us to prove that $\Ell^\bullet_G$ is in fact a cohomology theory, in spite of the fact that $\Map_G(P,-)$ does not preserve equivariant pushouts and the right Kan extension is not exact in general. \\

\noindent {\bf Outline of the paper.} In Section \ref{s2}, we introduce the basic geometric ingredients of the construction, which are the $2$-loop groupoid $\L^2_G$ of a compact Lie group $G$ and the notion of a frame in a principal bundle. In Section \ref{s3}, we introduce the abelian orbit category $\Orb$ of a topological groupoid $\G$ and, in the case where $\G$ is the 2-loop groupoid $\L^2_G$, we define the functor $\Fr$ sending an object $(P,T)$ to the subspace of frames $\Fr(T) \subset \Hom(\C,T)$ and prove some results concerning the structure of $\Fr(T)$. In Section \ref{s4}, for any given subfunctor $S$ of $\Hom(\C,-)$ (of which $\Fr$ is an example) we define a category $\Sh_S^{\C^\times}$ of $\C^\times$-equivariant holomorphic sheaves. This category is fibered over the (opposite) abelian orbit category of $\L^2_G$ and its fiber over $(P,T)$ is the category of $\C^\times$-equivariant $\O^\bullet_{S(T)}$-algebras. These categories serve as the targets for most of the functors that we construct. In Section \ref{s5}, we introduce complex analytic $T$-equivariant cohomology $\H^\bullet_T$ and prove several of its localisation properties, and we consider the reduced version $\widetilde{\H}^\bullet_T$ in order to formulate results about exactness and suspension. In Section \ref{s6}, we consider again an arbitrary topological groupoid $\G$ and we construct a functor 
\[
\Orbop \longrightarrow \Fun(\Topv{\G}^{\op},\Sh^{\C^\times})
\]
where $\Topv{\G}$ is the category of topological functors $\Fun(\G^{\op},\Top)$. (In fact, we define such a functor for each subfunctor $S$ of $\Hom(\C,-)$, because this is necessary for Proposition \ref{natural}, but the main case is $S = \Hom(\C,-)$.) We construct this functor because we find that it is easier than constructing the associated functor
\begin{align*}
\H^\bullet_\G: \Topv{\G}^{\op} &\longrightarrow \Gamma(\Orbop,\Sh^{\C^\times})\\
F & \longmapsto \left((P,T) \mapsto \H^\bullet_T(F(P) \right)
\end{align*}
which is what we are really interested in. Taking the latter functor, we define $\Ell^\bullet_G$ by first setting $\G = \L^2_G$ and precomposing with 
\begin{align*}
\Fin{G}^{\op} &\longrightarrow \Topv{\L^2_G}^{\op}  \\
X &\longmapsto \Map_G(-,X).
\end{align*}
We then postcompose with the functor
\[
\Gamma(\mathrm{Orb}^{\ab,\op}_{\L^2_G},\Sh^{\C^\times}) \longrightarrow \Gamma(\mathrm{Orb}^{\ab,\op}_{\L^2_G},\Sh_\Fr^{\C^\times}) 
\]
induced by restricting sheaves to subspaces of frames. The definition of $\Ell^\bullet_G$ is then followed by proofs of the basic properties of $\Ell^\bullet_G$ which constitute our main results. These results rely heavily on the properties of $\H^\bullet_T$ established in Section \ref{s5}, along with the results of Section \ref{locally} where we show that $\Map_G(P,X)$ is locally finite over $\Fr(T)$. We carry out much of the construction for a general topological groupoid $\G$ rather than the 2-loop groupoid $\L^2_G$ specifically because this will be useful in future applications. \\

\noindent {\bf Notation and conventions.} 
The notation $H \leq G$ denotes a closed subgroup of a topological group $G$. We write $\bT^2$ for the standard 2-dimensional torus $\R^2/\Z^2$. If $M$ is a smooth manifold, we write $\Diff(M)$ for the group of diffeomorphisms of $M$. All topological spaces are assumed to be compactly generated and weak Hausdorff. \\

\noindent {\bf Acknowledgements:}
We are indebted to Charles Rezk for the program that we pursue in this paper. We would like to thank him for generously sharing his ideas and for helpful conversations.  We also wish to thank Jack Davies, Nikolay Konovalov and Gerd Laures for conversations which helped in the writing of this paper.

\section{Basic geometric ingredients}\label{s2}

\subsection{The $2$-loop groupoid}

\begin{definition}
Let $G$ be a compact Lie group. The $2$-loop groupoid $\L^2_G$ of $G$ is the topological groupoid $\L^2_G$ whose objects are smooth principal $G$-bundles $P \to \Sigma$ over a smooth genus 1 surface $\Sigma$. A morphism $(P \to \Sigma) \to (P' \to \Sigma')$ in this category is a $G$-equivariant diffeomorphism
\[
\begin{tikzcd}
P \ar[r,"\sim"] \ar[d] & P' \ar[d] \\
\Sigma \ar[r,"\sim"] & \Sigma'
\end{tikzcd}
\]
of the total manifolds covering a diffeomorphism of the base manifolds. The topology on mapping spaces is the compact-open topology. 
\end{definition}

\begin{notation}
We will often denote a $G$-bundle $P \to \Sigma$ in $\L^2_G$ just by $P$, while the base manifold $\Sigma $ and the map $P \to \Sigma$ will be implicit. We denote the automorphism group of an object $P \to \Sigma$ by $\Aut(P)$. 
\end{notation}

\begin{remark}
An principal $G$-bundle $\pi: P \to \Sigma$ in $\L^2_G$ determines a homomorphism 
\[
\tilde{\pi}: \Aut(P) \to \Diff(\Sigma)
\]
of topological groups from the automorphism group of $P \to \Sigma$ into the diffeomorphism group of $\Sigma$. The kernel of $\tilde{\pi}$ is, by definition, the gauge group $\Gau(P)$ of $P$, and we have the exact sequence
\[
1 \longrightarrow \Gau(P) \longrightarrow \Aut(P) \xrightarrow{\tilde{\pi}} \Diff(\Sigma).
\]
\end{remark}

\begin{example}
Let $P := \bT^2 \times G \to \bT^2$ be the trivial bundle over the standard torus. In this case, $\Aut(P)$ is the semidirect product group $\Map(\bT^2,G) \rtimes \Diff(\bT^2)$, sometimes called the extended double loop group of $G$. The semidirect product structure is determined by the natural action of $\Diff(\bT^2)$ on $\Map(\bT^2,G)$ by precomposition. The gauge group in this case is $\Gau(P) = \Map(\bT^2,G)$, which acts on $P$ via the $G$-action on each fiber. 
\end{example}

\subsection{Frames}

The following notion of a frame was introduced in \cite{Rezktalk}.

\begin{definition}
Let $P \to \Sigma$ be an object in $\L^2_G$. A frame in $P$ is a homomorphism $\gamma: \C \to \Aut(P)$ of topological groups such that the induced action of $\C$ on $\Sigma$ is transitive.
\end{definition}

\begin{example}
Let $P := \bT^2 \times G \to \bT^2$ be the trivial $G$-bundle over the standard torus. The identification $\C \cong \R^2$ induces a translation action of $\C$ on $\bT^2 := \R^2/\Z^2$ with isotropy $\Z^2 \subset \C$. Suppose that $\C$ also acts smoothly on $G$ (e.g. by the trivial action). Then the diagonal action of $\C$ on $\bT^2 \times G$ defines a frame in $P$.
\end{example}

Let $\gamma$ be a frame in $P$ and consider the associated $\C$-action on $\Sigma$. Since the action is transitive and smooth, $\Sigma$ is a homogeneous $\C$-space with isotropy group $\Lambda_\gamma$. The isotropy group must be a lattice in $\C$ since $\C$ and $\Sigma$ both have real dimension $2$ and $\Sigma$ is a torus. Since $\Lambda_\gamma$ acts on $P$ by automorphisms covering the identity map of $\Sigma$, we have a commutative diagram
\begin{equation}\label{latty}
\begin{tikzcd}
\Lambda_\gamma \ar[d,hook] \ar[r] & \Gau(P) \ar[d,hook] \\
\C \ar[r,"\gamma"] \ar[d, two heads] & \Aut(P) \ar[d,"\tilde{\pi}"] \\
\C/\Lambda_\gamma \ar[r,"\tilde{\gamma}",dashed] & \Diff(\Sigma) 
\end{tikzcd}
\end{equation}
The transitive $\C$-action on $\Sigma$ induces a diffeomorphism of $\Sigma$ with the elliptic curve $\C/\Lambda_\gamma$, and the diffeomorphism is unique up to translation in $\C/\Lambda_\gamma$. 

\begin{proposition}\label{bijecto}
Let $P \to \Sigma$ be an object in $\L^2_G$. There is a bijection between the set $\Fr(P)$ of frames in $P$ and the set of triples $(\Lambda,\eta,\nu)$ where $\Lambda \leq \C$ is a lattice, $\eta: \C/\Lambda \cong \Sigma$ is a diffeomorphism (unique up to translation in $\C/\Lambda$) and $\nu$ is a flat connection on $P \to \Sigma$.
\end{proposition}

\begin{proof}
Let $\gamma$ be a frame for $P$. We have already explained how to associate a lattice $\Lambda = \Lambda_\gamma$ and a diffeomorphism $\eta$ to $\gamma$. It remains to show how $\gamma$ determines a flat connection $\nu$ in $P$, which we describe by its parallel transport operation. Fix a point $b \in \Sigma$ and a point $p \in P_b$ and consider the commutative diagram
\[
\begin{tikzcd}
\C \ar[d,two heads] \ar[r,"\gamma(-)(p)"] & P \ar[d,"\pi"] \\
\C/\Lambda_\gamma \ar[r,"\tilde{\gamma}(-)(b)"]& \Sigma.
\end{tikzcd}
\]
The composite of the left and bottom arrows is a local diffeomorphism, and so any path $\alpha$ in $\Sigma$ has a smooth lift $\widetilde{\alpha}$ to $\C$. The complex number $\widetilde{\alpha}(1)-\widetilde{\alpha}(0)$ does not depend on the choice of the lift or the choice of identification of $\C/\Lambda_\gamma$ with $\Sigma$, because any other such choices merely result in a translation of the lift $\widetilde{\alpha}$. It is straightforward to verify that the $G$-diffeomorphism
\[
\gamma(\widetilde{\alpha}(1)-\widetilde{\alpha}(0)): P \to P
\]
restricts to a $G$-diffeomorphism $P_{\alpha(0)} \to P_{\alpha(1)}$. Furthermore, if $\alpha$ is a contractible loop in $\Sigma$, then $\widetilde{\alpha}$ is a contractible loop in $\C$ since the homotopy lifts along $\C \twoheadrightarrow \Sigma$. In particular, $\widetilde{\alpha}(1)=\widetilde{\alpha}(0)$ for a contractible loop $\alpha$ and the associated connection is therefore flat. It is straightforward to show that the concatenation of paths is preserved.

Conversely, suppose that we are given a diffeomorphism $\eta: \Sigma \cong \C/\Lambda$ and a flat connection $\nu$ on $\pi:P \to \Sigma$. We define the associated frame $\gamma$ as follows. Let $(z,p) \in \C \times P$. The composite 
\[
\C \twoheadrightarrow \C/\Lambda \xrightarrow{ +\eta^{-1}(\pi(p))} \C/\Lambda \xrightarrow{\eta} \Sigma
\]
is a local diffeomorphism sending $0 \mapsto \pi(p)$, where the map in the middle is translation by $\eta^{-1}(\pi(p)) \in \C/\Lambda$. Let $\alpha$ be a path in $\C$ from $0$ to $z$. The image of $\alpha$ under the composite map above is a path in $\Sigma$ beginning at $\pi(p)$, and the connection determines a lift $\widetilde{\alpha}$ to $P$ beginning at $p$. We set $\gamma(z)(p) := \widetilde{\alpha}(1)$. This is well defined since, if $\alpha'$ is another path in $\C$ from $0$ to $z$, then the concatenation of $\alpha'$ with the path obtained by \lq going backward' along $\alpha$ is a contractible loop in $\C$. This defines a contractible loop in $\Sigma$ whose lift to $P$ by means of the flat connection $\nu$ results in a trivial parallel transport operation. It is straightforward to check that these constructions are inverse to each other. 
\end{proof}

\begin{remark}
Let $\gamma$ be a frame in $P$, let $\lambda \in \C^\times$ and write $m_\lambda: \C \to \C$ for scalar multiplication by $\lambda$. Then
\[
\C \xrightarrow{m_\lambda} \C \xrightarrow{\gamma} \Aut(P)
\]
is also a frame in $P$, since scaling by $\lambda$ preserves the image of $\gamma$, and therefore also preserves the transitivity of the $\C$-action on $\Sigma$. Furthermore, if $g: P' \to P$ is a morphism in $\L^2_G$, then
\[
\C \xrightarrow{\gamma} \Aut(P) \xrightarrow{g^{-1}(-) g} \Aut(P')
\]
is a frame in $P'$, since conjugation by the diffeomorphism $g$ preserves transitivity of the $\C$-action on $\Sigma$. 
\end{remark}

The following proposition is a straightforward extension of Proposition \ref{bijecto} and we omit the proof.

\begin{proposition}\label{actionn}
Let $\lambda \in \C^\times$ and let
\[
\begin{tikzcd}
P' \ar[r,"g"] \ar[d] & P \ar[d] \\
\Sigma' \ar[r,"\tilde{g}"] & \Sigma
\end{tikzcd}
\]
be a morphism in $\L^2_G$. Suppose that $\gamma$ is the frame in $P \to \Sigma$ corresponding to the triple $(\Lambda,\eta,\nu)$. Then 
\[
\C \xrightarrow{m_\lambda} \C \xrightarrow{\gamma} \Aut(P) \xrightarrow{g^{-1}(-) g} \Aut(P')
\]
is the frame in $P' \to \Sigma'$ corresponding to the triple $(\frac{1}{\lambda} \Lambda, \eta^{\lambda,g}, g^*\nu)$, where $g^*\nu$ is the connection defined by pulling back $\nu$ along the $G$-diffeomorphism $g: P' \to P$, and $ \eta^{\lambda,g}$ is the diffeomorphism defined by the commutative diagram
\[
\begin{tikzcd}
\C/\frac{1}{\lambda}\Lambda_{\gamma} \ar[r,"\tilde{m}_\lambda"] \ar[d,"\eta^{\lambda,g}"] & \C/\Lambda_\gamma \ar[d,"\eta"] \\
\Sigma' \ar[r,"\tilde{g}"] & \Sigma.
\end{tikzcd}
\]
All arrows in this diagram are diffeomorphisms and $\tilde{m}_\lambda$ is the isomorphism of elliptic curves induced by $m_\lambda$. Translating $\eta$ by $[z]$ corresponds to translating $\eta^{\lambda,g}$ by $[\frac{1}{\lambda}z]$.
\end{proposition}

The next result suggests that we classify frames in $P$ by reference to some torus in $\Aut(P)$ that they live inside. This is the motivation for the abelian orbit category of a topological groupoid, to be defined in the next section.

\begin{proposition}\label{framestori}
If $\gamma$ is a frame in $P$, then the closure $\widebar{\gamma(\C)}$ of its image in $\Aut(P)$ is compact. In particular, there exists a torus $T \leq \Aut(P)$ such that $\gamma(\C)$ is contained in $T$.
\end{proposition}

\begin{proof}
The action of $\widebar{\gamma(\C)}$ on $\Sigma$ is transitive and the isotropy group of any point $b \in \Sigma$ is equal to the closure of $\gamma(\Lambda)$ in $\Aut(P)$. The action of $\widebar{\gamma(\C)}$ on $b$ therefore yields a fiber bundle $\widebar{\gamma(\Lambda)} \to \widebar{\gamma(\C)} \to \Sigma$. Since $\Sigma$ is compact, the result follows if we can show that $\widebar{\gamma(\Lambda)}$ is compact. The action of $\widebar{\gamma(\Lambda)}$ on $\Sigma$ is trivial, and therefore $\widebar{\gamma(\Lambda)}\leq \Gau(P)$. Furthermore, the conjugation action of $\widebar{\gamma(\Lambda)}$ on $\gamma(\C)$ is trivial, so that $\widebar{\gamma(\Lambda)}$ is contained in the isotropy subgroup of the $\Gau(P)$-action on (the flat connection associated to) $\gamma$, which is known to be compact (see for example Prop. 2.4 in \cite{Nara}, for an argument which generalises easily to this context). Since $\widebar{\gamma(\C)}$ is closed, it is also compact.
\end{proof}

\section{The abelian orbit category of a topological groupoid}\label{s3}

In this section we define the abelian orbit category $\Orb$ associated to a topological groupoid $\G$, which was originally introduced in \cite{Rezktalk}. It is a generalisation of the abelian orbit category of a topological group in the sense that, if $\mathbf{B}G$ is the topological groupoid with just one object and automorphism group $G$, then $\mathrm{Orb}^{\ab}_{\mathbf{B}G}$ is the usual orbit category $\mathrm{Orb}^{\ab}_G$ with compact abelian isotropy. We will generally use $P$ for an object of $\G$ and $\Aut(P)$ for its automorphism group (since in our main application, the objects are principal bundles).

\subsection{Definition}

\begin{definition}
The abelian orbit category $\Orb$ of a topological groupoid $\G$ is the category whose objects are pairs $(P,T)$, where $P$ is an object in $\G$ and $T$ is a compact abelian subgroup of $\Aut(P)$. A morphism $(P,T) \to (P',T')$ is a natural transformation of functors
\[
\Map_{\G}(P,-)/T \to \Map_{\G}(P',-)/T'
\]
from $\G$ into $\Top$. The quotients are topological quotients by the continuous $T$-action given by precomposing with automorphisms of $P$.
\end{definition}

\begin{example}
If $\G = \mathbf{B}G$, then the orbit category $\mathrm{Orb}^{\ab}_{\G}$ is equivalent to the usual orbit category $\mathrm{Orb}^{\ab}_G$ of $G$ with compact abelian isotropy. The equivalence sends an object $(\pt,T)$ to $G/T$ and a natural transformation of functors
\[
\Map_{\mathbf{B}G}(\pt,-)/T \to \Map_{\mathbf{B}G}(\pt,-)/T'
\]
to its evaluation on the single object of $\mathbf{B} G$, which is a $G$-equivariant continuous map $G/T \to G/T'$. 
\end{example}

The following proposition is a straightforward application of the Yoneda lemma. It provides us with concrete representatives of morphisms in $\Orb$.

\begin{proposition}\label{yoneda}
Let $(P,T)$ and $(P',T')$ be objects in $\Orb$. There is a bijection between the set of morphisms from $(P,T)$ to $(P',T')$ and the set $(\Map_\G(P',P)/T')^T$ of $T'$-orbits that are fixed by $T$ (i.e. orbits that are represented by some $g: P' \to P$ satisfying $g^{-1}T g \leq T'$). 

The bijection sends a natural transformation $\phi: (P,T) \to (P',T')$ to $\phi_P([\id_P]_T)$, where $[\id_P]_T$ denotes the $T$-orbit of the identity morphism $\id_P \in \Aut(P)$. The inverse map sends a $T'$-orbit $[g]_{T'}$ represented by $g: P \to P'$ to the natural transformation induced by precomposition with $g$.
\end{proposition}

\begin{proof}
Consider the composite map
\begin{align*}
\Nat(\Map_\G(P,-)/T, \Map_\G(P',-)/T') &\hookrightarrow  \Nat(\Map_\G(P,-), \Map_\G(P',-)/T') \\
&\cong \Map_{\G}(P',P)/T'
\end{align*}
where the inclusion is given by pullback along the $T$-quotient and the bijection $\phi \mapsto \phi_P(\id_P)$ is the bijection of the Yoneda lemma. The image of the first map is the subset of transformations $\Map(P,-) \to \Map(P',-)/T'$ that factor through the quotient by $T$. It will suffice to show that the image of the composite map is the set of $T$-fixed points of $\Map_\G(P',P)/T'$. Suppose that a natural transformation $\phi: \Map_\G(P,-) \to \Map_\G(P',-)/T'$ factors through the quotient by $T$. Then the component of $\phi_P$ of $\phi$ at $P$ factorises 
\[
\phi_P: \Map_{\G}(P,P) \to \Map_{\G}(P,P)/T \to \Map_{\G}(P',P)/T'
\]
through the quotient. In particular, $\phi_P(t) = \phi_P(\id_P)$ for all $t \in T \leq \Aut(P) =  \Map_{\G}(P,P)$, and so by naturality of $\phi$ we have $t \cdot \phi_P(\id_P) = \phi_P(t) =\phi_P(\id_P)$ for all $t \in T$. Therefore $\phi_P(\id_P)$ is fixed by $T$.

Conversely, let $[g]_{T'} \in  \Map_{\G}(P',P)/T'$ be a $T'$-orbit represented by $g:P' \to P$. Then $[g]_{T'}$ is fixed by $T$ if and only if
\[
[t\circ g]_{T'} = [g]_{T'} \quad \text{for all}\quad t \in T.
\]
This is true if and only if, for each $t \in T$, there exists $t' \in T'$ such that $t \circ g = g \circ t'$, which is true if and only if
\[
g^{-1} T g \leq T'.
\]
The natural transformation $\phi_g: \Map_{\G}(P,-) \to \Map_{\G}(P',-)/T'$ which corresponds to $[g]_{T'}$ under the bijection
\[
\Nat(\Map_\G(P,-), \Map_\G(P',-)/T')\cong \Map_{\G}(P',P)/T'
\]
above is induced by precomposition by $g$. Assuming that $[g]_{T'}$ is fixed by $T$, we want to show that $\phi_g$ factors through the quotient by $T$. Let $P''$ be an object in $\G$ and let $f$ and $f'$ be in the same $T$-orbit of the $T$-space $\Map_\G(P,P'')$. Then there exists $t \in T \leq \Aut(P)$ such that $f' = f \circ t$. Therefore 
\[
\phi_{g,P''}(f') = [f' \circ g]_{T'} = [f\circ t \circ g]_{T'} =  [f]\circ [t \circ g]_{T'} =  [f]\circ [ g]_{T'} = [f \circ g]_{T'} =\phi_{g,P''}(f)
\]
which means that $\phi_g$ factors through the $T$-quotient, so we are done.
\end{proof}

\begin{remark}
We may use this result to enrich the category $\Orb$ over $\Top$ by transferring the natural topology on $(\Map_\G(P',P)/T')^T$ to the set of morphisms $\{(P,T) \to (P',T')\}$, since the composition of morphisms in $\Orb$ is induced by composition in $\G$.  
\end{remark}

\begin{notation}
Write $W(\Aut(P)/T)$ for the automorphism group of an object $(P,T)$ in $\Orb$. This notation is supposed to suggest the idea of a Weyl group, as well as to clearly distinguish the automorphism group of $(P,T)$ in $\Orb$ from the automorphism group $\Aut(P)$ of $P$ in $\G$.
\end{notation}

\begin{remark}
By Proposition \ref{yoneda}, $W(\Aut(P)/T)$ is isomorphic to $(\Aut(P)/T)^T$, which is equal to the quotient group $N(P,T) /T$ of the normaliser of $T \leq \Aut(P)$ by $T$. In other words, $W(\Aut(P)/T)$ is indeed isomorphic to the Weyl group of $T$ in $\Aut(P)$. 
\end{remark}

The next corollary follows immediately from Proposition \ref{yoneda}.

\begin{corollary}
Let $\G$ be a topological groupoid. A choice of representative object for each isomorphism class of objects in $\G$ determines an equivalence of categories
\[
\Orb \simeq \coprod_{[P] \in \pi_0 \G} \mathrm{Orb}^{\ab}_{\Aut(P)} 
\]
where $\mathrm{Orb}^{\ab}_{\Aut(P)} $ is the orbit category of $\Aut(P)$ with compact abelian isotropy. 
\end{corollary}

\begin{definition}\label{maximal}
An object $(P,T)$ in $\Orb$ is {\em maximal} if the identity component of $T$ is maximal among tori in $\Aut(P)$. 
\end{definition}

\begin{lemma}
An object $(P,T)$ in $\Orb$ is maximal if and only if $\W(\Aut(P)/T)$ is discrete. 
\end{lemma}

\begin{proof}
Suppose that $T$ is a maximal torus in $\Aut(P)$, and consider the exact sequence
\[
1 \longrightarrow T \longrightarrow N(P,T) \longrightarrow W(\Aut(P)/T) \longrightarrow 1.
\]
Assume that $W(\Aut(P)/T)$ is not discrete. Then there exists a nontrivial one-parameter subgroup $\alpha: \R \to W(\Aut(P)/T)$. The closure of the preimage of $\alpha(\R)$ in $N(P,T)$ is then a compact abelian subgroup of $\Aut(P)$ whose maximal torus contains $T$ as a proper subgroup, which contradicts the maximality of $T$. The other direction is clear.
\end{proof}

\subsection{Examples}

We compute $W(\Aut(P)/T)$ for some examples of objects $(P,T)$ in $\mathrm{Orb}^{\ab}_{\L^2_G}$ for a compact Lie group $G$.

\begin{example}\label{trivialexample}
Let $A$ be a maximal torus in $G$. If $P$ is the trivial bundle $\bT^2 \times G$, then 
\[
T := A \times \bT^2 \leq \Map(\bT^2,G) \rtimes \Diff(\bT^2) = \Aut(P)
\]
is a maximal torus in $\Aut(P)$, where $\bT^2$ is identified with translations in $\Diff(\bT^2)$ and $A$ is identified with the group of constant maps $\bT^2 \to \pt \to G$ taking values in $A$. In this case, the automorphism group $W(\Aut(P)/T)$ of $(P,T)$ is the group 
\[
\Hom(\bT^2,A) \rtimes (W(G/A) \times \GL_2(\Z)) 
\]
sometimes called the elliptic Weyl group of $G$. Here 
\[
W(G/A) \hookrightarrow \Map(\bT^2,G)
\]
is identified with the subgroup of constant maps taking values in the (finite) Weyl group of $A \leq G$,
\[
\Hom(\bT^2,A) \leq \Map(\bT^2,A) \hookrightarrow \Map(\bT^2,G)
\]
is the subgroup of group homomorphisms with values in $A$ and 
\[
\GL_2(\Z) \leq \Diff(\bT^2)
\]
is the subgroup of diffeomorphisms of $\bT^2$ that are induced by linear automorphisms of $\R^2$. The action of $W(G/A)$ on $\Hom(\bT^2,A)$ is given by postcomposition with the Weyl action and the action of $\GL_2(\Z)$ on $\Hom(\bT^2,A)$ is precomposition by diffeomorphisms of $\bT^2$.
\end{example}

\begin{example}
Let $G$ be a finite group, let $g,h$ be commuting elements in $G$ and let $P_{g,h}$ be the principal $G$-bundle over $\bT^2$ defined by the homomorphism $\Z^2 \to G$ sending $(1,0) \mapsto g$ and $(0,1) \mapsto h$. Write $C_G(g,h)$ for the centraliser of the pair $g,h$ in $G$. The subspace 
\[
\R^2 \times C_G(g,h) / \langle (-1,0,g)(0,-1,h)\rangle  \subset \R^2 \times G / \langle (-1,0,g)(0,-1,h)\rangle =: P_{g,h},
\]
has a group structure induced by that of $\R^2$ and $G$ and acts smoothly and faithfully on $P_{g,h}$. The automorphism group $\Aut(P_{g,h})$ is homotopy equivalent to
\[
(\R^2 \times C_G(g,h) / \langle (-1,0,g)(0,-1,h)\rangle) \rtimes \Gamma(g,h)  
\]
where $\Gamma(g,h)$ is the subgroup 
\[
\Gamma(g,h) = \left\{ \begin{pmatrix} a & b \\ c & d \end{pmatrix} \in \GL_2(\Z) \, \bigg| \, (g^ah^c ,g^bh^d) = (g,h) \right\} 
\]
of linear automorphisms in $\Diff(\bT^2)$. The semidirect product structure is induced by the multiplication action of $\Gamma(g,h) $ on $\R^2$ and the trivial action on $C_G(g,h)$. The subgroup 
\[
T_{g,h} := \R^2 / |g|\Z \times |h|\Z \quad \leq \quad \Aut(P_{g,h})
\]
is a maximal torus, and $(P_{g,h},T_{g,h})$ is a maximal object in $\mathrm{Orb}^{\ab}_{\L^2_G}$ with automorphism group 
\[
W(\Aut(P_{g,h})/T_{g,h}) = \Gamma(g,h) \times C_G(g,h).
\]
If $N$ is the order of $G$, then $\Gamma(g,h)$ contains the kernel of the projection $\GL_2(\Z) \twoheadrightarrow \GL_2(\Z/N)$, which is the principal congruence subgroup $\Gamma(N)$. Thus, $\Gamma(g,h)$ is also a congruence subgroup of $\GL_2(\Z)$. 
\end{example}

\begin{example}
Let $G = \Z/N$. We have  
\[
\Gamma(1,0) = \left\{   A \in \GL_2(\Z) \,\bigg|\, A = \begin{pmatrix}
1 & * \\ 
0 & \det A
\end{pmatrix} \mod N\right\}
\]
and
\[
\Gamma(0,1) = \left\{   A \in \GL_2(\Z) \,\bigg| \, A = \begin{pmatrix}
\det A & 0 \\ 
* & 1
\end{pmatrix} \mod N\right\}.
\]
Note that $\Gamma(1,0) \cap \SL_2(\Z)$ is the congruence subgroup of $\SL_2(\Z)$ usually denoted by $\Gamma_1(N)$.
\end{example}

\begin{example}
Let $G = \Z/N \times \Z/N$. A matrix $A \in \GL_2(\Z)$ fixes a pair 
\[
((m_1,m_2),(n_1,n_2)) \in (\Z/N \times \Z/N) \times (\Z/N \times \Z/N)
\]
if and only if $A$ fixes both $(m_1,n_1)$ and $(m_2,n_2)$ under matrix multiplication. So, by the previous example, we have
\[
\Gamma((1,0),(0,1)) = \left\{   A \in \GL_2(\Z) \,\middle|\, A = \begin{pmatrix}
1 & 0 \\ 
0 & 1
\end{pmatrix} \mod N\right\} = \Gamma(N).
\]
\end{example}

\subsection{Classifying frames}\label{actING}

It follows from Proposition \ref{framestori} that, for any frame $\gamma: \C \to \Aut(P)$ in an object $P$ of $\L^2_G$, there exists a compact abelian subgroup $T \leq \Aut(P)$ such that the image of $\gamma$ is contained in $T$. In this section we use this fact to classify frames over the abelian orbit category of $\L^2_G$.

It is well known that a compact abelian topological group has a smooth structure, unique up to diffeomorphism, and that any continuous homomorphism of Lie groups is smooth. We will use these facts in the sequel without specific reference to them.

\begin{remark}
If $T$ is a compact abelian group, then the space $\Hom(\C,T)$ of homomorphisms of Lie groups is a complex vector space. Addition is defined pointwise using the abelian group structure on $T$ and scaling is defined by precomposition with the conjugate scalar action on $\C$, so that 
\[
(\lambda \cdot \gamma)(z) := \gamma(\overline{\lambda}z) \quad \text{for all} \ \lambda \in \C^\times, z \in\C.
\] 
The complex vector space $\Hom(\C,T)$ may be identified with the space $\Hom_\R(\C,\Lie(T))$ of $\R$-linear maps by differentiating at zero, and thence with the complex Lie algebra $\Lie(T) \otimes_\R \C$ of $T$. 
\end{remark}

\begin{definition}
An $n$-dimensional $\C^\times$-equivariant complex manifold is a $\C^\times$-equivariant topological space $M$ equipped with a maximal atlas of charts $\C^n \supset U \xrightarrow{\phi_i} M$ such that each chart is $\C^\times$-equivariant (with the usual scalar action on $\C^n$) and transition functions are biholomorphic. A morphism $M \to M'$ of $\C^\times$-equivariant complex manifolds is a $\C^\times$-equivariant holomorphic map. Let $\CxMfd^{\C^\times}$ denote the category of $\C^\times$-equivariant complex manifolds with these morphisms. 
\end{definition}

\begin{definition}
Let $\G$ be a topological groupoid. Define the functor
\[
\Hom(\C,-): \Orb \longrightarrow \CxMfd^{\C^\times}
\]
which sends an object $(P,T)$ to the complex vector space $\Hom(\C,T)$ and a morphism $\phi:(P,T) \to (P',T')$ to the $\C$-linear map $\Hom(\C,T) \to \Hom(\C,T')$ given by postcomposition with the homomorphism $T \to T'$ defined by $t \mapsto g^{-1}t g$, where $g:P' \to P$ is any morphism in $\G$ that represents $\phi$. 
\end{definition}

\begin{notation}
Let $\G$ be a topological groupoid. To simplify notation, we will write $\phi$ not only for a morphism $\phi:(P,T) \to (P',T')$ in $\Orb$ but also for the induced homomorphism $\phi: T \to T'$ and for the induced complex linear map $\phi: \Hom(\C,T) \to \Hom(\C,T')$. We will often abbreviate $\Hom(\C,T)$ to $T^\C$.
\end{notation}

In the remainder of this subsection, we set $\G = \L^2_G$.

\begin{definition}
Let $(P,T)$ be an object in $\mathrm{Orb}_{\L^2_G}^{\ab}$. We write $\Fr(T)$ for the subspace of $T^\C$ consisting of frames with image in $T$.
\end{definition}

\begin{lemma}\label{complex}
Let $(P,T)$ be an object in $\mathrm{Orb}_{\L^2_G}^{\ab}$. The space $\Fr(T)$ is a $\C^\times$-equivariant complex open submanifold of $T^\C$.
\end{lemma}

\begin{proof}
It suffices to show that $\Fr(T)$ is a $\C^\times$-equivariant open subset of $T^\C$, from which it follows that $\Fr(T)$ automatically has the structure of a complex submanifold of $T^\C$. Let $\pi: P \to \Sigma$ be the underlying principal bundle and $\tilde{\pi}$ the induced homomorphism of topological groups $\Aut(P) \to \Diff(\Sigma)$. The image $\tilde{T}:= \tilde{\pi}(T)$ of $T$ is also compact abelian, and $T$ and $\tilde{T}$ both have smooth structures such that the restriction of $\tilde{\pi}$ to $T$ is a homomorphism of Lie groups. Let $\gamma: \C \to T$ be a homomorphism of Lie groups. The derivative at zero of the composite $\tilde{\pi}|_T \circ \gamma$ is an $\R$-linear map $\C \to \Lie(\tilde{T})$, and $\gamma$ is a frame if and only if this map has nonzero determinant. Therefore the subspace $\Fr(T) \subset T^\C$ of frames is open. To see that the $\C^\times$-action on $\Hom(\C,T)$ preserves the subspace $\Fr(T)$, note that the $\C^\times$-action on a frame $\gamma$ preserves the image $\gamma(\C) \leq T$ and that it is only the image of $\gamma$ which determines whether the induced action on $\Sigma$ is transitive or not. 
\end{proof}

\begin{lemma}\label{subb}
The mapping which sends $(P,T)$ to $\Fr(T)$ and a morphism $\phi:(P,T) \to (P',T')$ to $\gamma \mapsto \phi \circ \gamma$ defines a functor
\[
\Fr:\mathrm{Orb}^{\ab}_{\L^2_G} \longrightarrow \CxMfd^{\C^\times}.
\]
\end{lemma}

\begin{proof}
This is partially implied by Proposition \ref{complex}. It remains to show that the mapping defined on morphisms sends frames to frames. Let $\gamma: \C \to T$ be a frame in $P$ and let $\phi:(P,T) \to (P',T')$ be a morphism in $\Orb$. By Proposition \ref{yoneda}, there exists a diffeomorphism $g: P' \to P$ so that the homomorphism $\phi:T \to T'$ of Lie groups is given by conjugating $\gamma$ with $g$. Since $g$ is a diffeomorphism and $\gamma$ is a frame, the subgroup $g^{-1}\gamma(\C) g \leq T'$ must act transitively on $\Sigma'$. Therefore $\phi \circ \gamma$ is a frame in $P'$.
\end{proof}

\begin{remark}
We continue to use $\tilde{T}$ to denote the image of a compact abelian subgroup $T \leq \Aut(P)$ under the canonical homomorphism $\tilde{\pi}: \Aut(P) \to \Diff(\Sigma)$. Let $\Gau(T)$ denote the intersection $T \cap \Gau(P)$, so that we have an exact sequence of compact abelian groups
\[
1 \to \Gau(T) \to T \to \tilde{T} \to 1.
\]
This sequence is natural in $T$. In particular, it is $W(\Aut(P)/T)$-equivariant, with $W(\Aut(P)/T)$ acting on $\tilde{T}$ via the homomorphism
\[
W(\Aut(P)/T) \to W(\Diff(\Sigma)/\tilde{T})
\]
induced by $\tilde{\pi}: \Aut(P) \to \Diff(\Sigma)$. Applying $\Hom(\C,-)$ yields an $\C^\times \times W(\Aut(P)/T)$-equivariant exact sequence of complex vector spaces. Let $\Fr(\tilde{T}) \subset \Hom(\C,\tilde{T})$ be the subspace of homomorphisms which induce a smooth transitive action on $\Sigma$. Then, by definition, $\Fr(T)$ fits into the $\C^\times \times W(\Aut(P)/T)$-equivariant pullback square
\[
\begin{tikzcd}
\Fr(T) \ar[r,hook] \ar[d,two heads]& \Hom(\C,T) \ar[d,two heads] \\
\Fr(\tilde{T}) \ar[r,hook] & \Hom(\C,\tilde{T}).
\end{tikzcd}
\]
This exhibits $\Fr(T)$ as a $\C^\times \times W(\Aut(P)/T)$-equivariant bundle of affine spaces with underlying vector space $\Hom(\C,\Gau(T))$. 
\end{remark}

\begin{proposition}\label{sectionz}
To each frame $\gamma \in \Fr(T)$ is associated a continuous global section $s_\gamma$ of the bundle $\Fr(T) \twoheadrightarrow \Fr(\tilde{T})$ such that $s_\gamma(\tilde{\gamma}) = \gamma$. The induced map from the set $\Fr(T)$ into the set of sections of $\Fr(T) \twoheadrightarrow \Fr(\tilde{T})$ is $\C^\times \times W(\Aut(P)/T)$-equivariant.
\end{proposition}

\begin{proof}
Fix $\gamma,\gamma' \in \Fr(T)$ and choose $b \in \Sigma$. Associated to $\gamma$ and $\gamma'$ are diffeomorphisms
\[
\C/\Lambda_{\gamma'} \xrightarrow{\tilde{\gamma}'(-)(b)} \Sigma \xleftarrow{\tilde{\gamma}(-)(b)} \C/\Lambda_{\gamma}.
\]
We denote the composite of the first arrow with the inverse of the second by $\tilde{\eta}_{\gamma',\gamma}$. It is easily verified that $\tilde{\eta}_{\gamma',\gamma}$ is an isomorphism of underlying Lie groups that does not depend on the choice of $b \in \Sigma$. Let $\eta_{\gamma',\gamma}: \C \to \C$ be the map of Lie algebras induced by $\tilde{\eta}_{\gamma',\gamma}$, so that it fits into the commutative diagram
\[
\begin{tikzcd}
\C \ar[d,two heads]\ar[rr,"\eta_{\gamma',\gamma}"] && \C \ar[d,two heads] \ar[r,"\gamma"] & T \ar[d,two heads]  \\
\C/\Lambda_{\gamma'} \ar[rr,"\tilde{\eta}_{\gamma',\gamma}"] &&\C/\Lambda_{\gamma} \ar[r,"\tilde{\gamma}"] & \tilde{T}
\end{tikzcd}
\]
We define $s_\gamma: \Fr(\tilde{T}) \to \Fr(T)$ to be the map sending $\tilde{\gamma}'$ to be the composite of the upper two arrows. Thus $s_\gamma$ is continuous, since the $\R$-linear map $\eta_{\gamma',\gamma}$ varies continuously with $\gamma'$. Furthermore, $s_\gamma$ is a section of the fiber bundle $\Fr(T) \to \Fr(\tilde{T})$ because the composite of the lower two arrows in the diagram above is equal to $\tilde{\gamma}'$ (noting that we may choose any $b \in \Sigma$ in the construction of $\tilde{\eta}_{\gamma',\gamma}$). Finally, by using Proposition \ref{actionn} in conjunction with the diagram above it is straightforward to verify the equation
\begin{equation}\label{acc}
s_{g^{-1}\gamma(\widebar{\lambda}z)g}(\tilde{g}^{-1}\tilde{\gamma}'([\widebar{\lambda}z]) \tilde{g}) = g^{-1} s_\gamma(\tilde{\gamma}')(\widebar{\lambda}z)g
\end{equation}
which implies the final claim in the proposition.
\end{proof}

\begin{corollary}\label{triv}
Each frame $\gamma \in \Fr(T)$ determines a trivialisation
\begin{align*}
\Fr(T) &\xrightarrow{\sim} \Fr(\tilde{T}) \times \Hom(\C,\Gau(T)) \\
\gamma' &\mapsto (\tilde{\gamma}', \gamma' - s_\gamma (\tilde{\gamma}')) 
\end{align*}
of the bundle $\Fr(T) \twoheadrightarrow \Fr(\tilde{T})$, where the subtraction takes place in the vector space $\Hom(\C,T)$. The trivialisation sends $\gamma$ to $(\tilde{\gamma},0)$.
\end{corollary}

\begin{remark} 
The trivialisation in Corollary \ref{triv} is non-equivariant, regardless of the choice of $\gamma$ (where the right hand side is equipped with its natural diagonal action of $\C^\times \times W(\Aut(P)/T)$). 
\end{remark}

\begin{example}\label{freg}
Suppose that $A \leq G$ is a torus in $G$ and that $P$ is the trivial bundle $\bT^2 \times G \twoheadrightarrow \bT^2$. Let $T$ be the torus 
\[
T := A \times \bT^2_{transl} \leq \Map(\bT^2,G) \rtimes \Diff(\bT^2) = \Aut(P)
\]
consisting of constant maps $\bT^2 \to A \leq G$ and translations of $\bT^2$. Let $\gamma: \C \to A \times \bT^2$ be a frame whose projection to $A$ is trivial. We have
\[
\Gau(T) := T \cap \Gau(P) = (A \times \bT^2) \cap \Map(\bT^2,G) = A.
\]
Consider the section $s_{\gamma}$ defined in Proposition \ref{sectionz}. For each frame $\gamma'$, the section returns a frame $s_{\gamma} (\tilde{\gamma}') \in \Fr(A\times \bT^2)$ in the same fiber as $\gamma'$. Our choice of $\gamma$ means that the projection of $s_{\gamma_\tau} (\tilde{\gamma}')$ to $A$ is trivial, while its projection to $\bT^2$ is $\tilde{\gamma}'$. The trivialisation given by Corollary \ref{triv} is then
\begin{align*}
\Fr(A \times \bT^2) &\xrightarrow{\sim} \Fr(\bT^2) \times \Hom(\C,A) \\
\gamma' &\mapsto (\gamma'_{\bT^2}, \gamma_A')
\end{align*}
where $\gamma'_{\bT^2}, \gamma_A'$ denote the projections of $\gamma'$ to $\bT^2$ and $A$ respectively. Notice that the trivialisation does not depend on $\gamma_{\bT^2}$. 
\end{example}

\begin{lemma}\label{inject}
Let $p \in P$ and let $A_p(T)$ denote the (compact abelian) image of $\Gau(T)$ under the homomorphism
\[
\Phi_p: \Gau(P) \to G
\]
sending $f$ to the unique element $g \in G$ such that $f(p) = p\cdot g$. If $\Fr(T)$ is not empty, then the restriction
\[
\Phi_p: \Gau(T) \xrightarrow{\sim} A_p(T)
\]
to $\Gau(T)$ is an isomorphism of Lie groups.
\end{lemma}

\begin{proof}
It suffices to show that if $\Fr(T)$ is not empty then $\Phi_p$ is injective on $\Gau(T)$. Let $\gamma$ be a frame in $T$. Then for any $p' \in P$ there exist $z \in \C$ and $g \in G$ such that $p' = \gamma(z)(p)\cdot g$. Let $f \in \Gau(T)$. Since $f$ lies in $T$, it commutes with $\gamma(z) \in T$ for all $z \in \C$. Thus
\begin{align*}
f(p') &= f(\gamma(z)(p)\cdot g) \\
&= \gamma(z)(f(p))\cdot g \\
&= \gamma(z)(p) \cdot \Phi_p(f) g 
\end{align*}
So $f$ is determined by $\Phi_p(f)$ and the map is injective.
\end{proof}

\begin{proposition}\label{triv2}
A choice of $\gamma \in \Fr(T)$, a basis $(t_1,t_2)$ for $\Lambda_\gamma$ and a point $p \in P$ induces a homeomorphism
\[
\Fr(T) \cong \X \times \Hom(\R,A_p(T)) \times \Hom(\R,A_p(T))
\]
where $\X \subset \C^2$ denotes the subspace of pairs $(t_1,t_2)$ such that $\R t_1 + \R t_2 = \C$ and $A_p(T)$ is the compact abelian image of $\Gau(T)$ in $G$ under $\Phi_p$.
\end{proposition}

\begin{proof}
There are isomorphisms of bundles
\begin{align*}
\Fr(T) &\cong \Fr(\tilde{T}) \times \Hom(\C,\Gau(T)) \\
&\cong \X \times \Hom(\C,\Gau(T))\\
&\cong  \X \times \Hom(\R^2,\Gau(T))\\
&\cong \X \times \Hom(\R^2,A_p(T)) \\
&\cong \X \times \Hom(\R,A_p(T)) \times \Hom(\R,A_p(T)).
\end{align*}
The first is the isomorphism of Corollary \ref{triv}, which is induced by the choice of $\gamma \in \Fr(T)$. The second isomorphism arises via pullback along the homeomorphism
\begin{align*}
\Fr(\tilde{T}) &\xrightarrow{\sim} \X \\
\tilde{\gamma}' &\mapsto (t_1',t_2') := (\eta^{-1}_{\gamma',\gamma}(t_1),\eta^{-1}_{\gamma',\gamma}(t_2))
\end{align*}
and is induced by the choice of basis $(t_1,t_2)$ for $\Lambda_\gamma$. The third is produced by applying $\Hom_{/\X}(-,\X \times \Gau(T))$ to the homeomorphism of bundles
\begin{align*}
\X \times \R^2 &\xrightarrow{\sim} \X \times \C \\
(t_1',t_2',s_1,s_2) &\mapsto (t_1',t_2',s_1t_1'+s_2t_2')
\end{align*}
over $\X$. The fourth is induced by the choice of $p \in P$ and the associated map $\Phi_p: \Gau(T) \xrightarrow{\sim} A_p(T)$, which is an isomorphism by Lemma \ref{inject} since $\Fr(T)$ is nonempty. The last isomorphism is clear.
\end{proof}

\begin{remark}
Explicitly, the map of Proposition \ref{triv2} sends 
\begin{equation}\label{formula}
\gamma' \mapsto (t_1',t_2', s\mapsto \Phi_p(\gamma'(st_1')\gamma(st_1)^{-1}), s\mapsto \Phi_p(\gamma'(st_2')\gamma(st_2)^{-1})) 
\end{equation}
where $(t_1',t_2') \in \X$ is the basis for $\Lambda_{\gamma'}$ corresponding to $(t_1,t_2)$ via $\eta_{\gamma',\gamma}$. The inverse map sends $(t_1',t_2',\phi_1,\phi_2)$ to the frame determined by
\[
(z, p) \mapsto \gamma(s_1t_1+s_2t_2)(p) \cdot \phi_1(s_1)\phi_2(s_2)
\]
where $(s_1,s_2)$ is the unique pair of real numbers such that $z = s_1t_1'+s_2t_2'$. 
\end{remark}

\section{Categories of sheaves over the orbit category}\label{s4}

\begin{definition}
Let $X$ be a complex manifold equipped with holomorphic $\C^\times$-action $a: \C^\times \times X \to X$. Let $p: \C^\times \times X \to X$ denote the projection. Write $\Sh(X)^{\C^\times}$ for the category of $\C^\times$-equivariant, graded sheaves of rings over $X$. Its objects are pairs $(\J^\bullet,\alpha)$, where $\J^\bullet$ is a sheaf
\[
\J^\bullet: \Open(X)^{\op} \to \Ring^\bullet
\]
with values in $\Z$-graded rings and $\alpha$ is a natural isomorphism $a^* \J^\bullet\to p^* \J^\bullet$ satisfying the natural cocycle condition over $\C^\times \times \C^\times \times X$. A morphism of such sheaves is a natural transformation $\psi: \J^\bullet\to \J'^\bullet$ such that
\[
\begin{tikzcd}
a^*\J^\bullet\ar[d,"\alpha"]\ar[r,"a^* \psi"] & a^*\J'^\bullet \ar[d,"\alpha'"] \\
p^* \J^\bullet\ar[r,"p^* \psi"] & p^*\J'^\bullet
\end{tikzcd}
\]
commutes. 
\end{definition}

\begin{notation}
Write $\omega$ for the complex vector space $(\R/\Z)^\C := \Hom(\C,\R/\Z)$ and write $\omega^{\otimes k}$ for the $k$-fold tensor product of $\omega$ if $k$ is positive, the $-k$-fold tensor product of the dual $\omega^*$ if $k$ is negative, and $\C$ if $k = 0$.
\end{notation}

\begin{definition}\label{otsheaf}
Let $X$ be a complex manifold. We write $\O^\bullet_X$ for the $\C^\times$-equivariant graded sheaf over $X$ defined on an open $U \subset X$ by 
\[
\O^\bullet_X(U) := \bigoplus_{n\in\Z} \O^{n}_X(U)
\]
where
\[
\O^{2j}_X(U) = \{f: U \to \omega^{\otimes j} \, |\, f \text{ holomorphic }\} \quad \text{and} \quad \O^{2j+1}_X(U) = 0
\]
for all $j \in \Z$. Thus, $\O^\bullet_X(U) $ is graded ring with addition given by pointwise addition of functions and a product given by the tensor product $fg:= f \otimes g$, which is an element in $\O_X^{j+k}(U)$ whenever $f \in \O^{j}_X(U)$ and $g\in \O^{k}_X(U)$. If $X$ is a $\C^\times$-equivariant complex manifold, then $\O^\bullet_X$ has the $\C^\times$-equivariant structure which is given over $U \subset \C^\times \times X$ on the $2j$th piece by 
\begin{align*}
 (a^*\O_X^{2j})(U) &\to (p^*\O_X^{2j})(U) \\
f(\lambda,z) &\mapsto \lambda^{-j} f(\lambda,z)
\end{align*}
where $a, p : \C^\times \times X \to X$ are the action and projection maps, respectively.
\end{definition}

\begin{definition}
We write $\Sh(X,\O^\bullet)^{\C^\times}$ for the category of $\C^\times$-equivariant $\O^\bullet_X$-algebras. This is the category of objects in $\Sh(X)^{\C^\times}$ equipped with a structure morphism from $\O^\bullet_X$. A morphism in $\Sh(X,\O^\bullet)^{\C^\times}$ is a morphism in $\Sh(X)^{\C^\times}$ which is compatible with structure morphisms.
\end{definition}

\begin{remark}
Suppose that $\phi: X \to Y$ is a $\C^\times$-equivariant holomorphic map of $\C^\times$-equivariant complex manifolds. The pullback of sheaves along $\phi$ yields a functor
\[
\Sh(Y,\O^\bullet)^{\C^\times} \xrightarrow{\phi^*} \Sh(X,\O^\bullet)^{\C^\times} 
\]
\end{remark}

\begin{definition}
Define the category $\Sh^{\C^\times}$ over $\Orbop$ whose objects are triples $(P,T,\J^\bullet)$, where $(P,T)$ is an object in $\Orb$ and $\J^\bullet$ is an object in $\Sh(T^\C,\O^\bullet)^{\C^\times}$. A morphism is a pair $(\phi,h):(P,T,\J^\bullet) \to (P',T',\J'^\bullet)$, where $\phi:(P',T') \to (P,T)$ is a morphism in $\Orb$ and $h: \phi^*\J^\bullet\to \J'^\bullet$ is a morphism in $\Sh(T^\C,\O^\bullet)^{\C^\times}$. The composition of two morphisms 
\[
(\phi,h):(P,T,\J^\bullet) \to (P',T',\J'^\bullet) \quad \text{and} \quad (\phi',h'):(P',T',\J'^\bullet) \to (P'',T'',\J''^\bullet)
\]
is the pair consisting of $\phi \circ \phi': (P'',T'') \to (P,T)$ and the composite map
\[
(\phi \circ \phi')^*\J^\bullet \cong \phi'^*(\phi^*\J^\bullet) \xrightarrow{\phi'^*(h)} \phi'^*\J'^\bullet \xrightarrow{h'} \J''^\bullet.
\]
\end{definition}

The canonical projection functor 
\[
\Sh^{\C^\times} \longrightarrow \Orbop
\]
is a cocartesian fibration, which is to say that for each morphism in the base category there is an cocartesian morphism in $\Sh^{\C^\times}$ lying over it. We give the definition of an cocartesian morphism in this context below, followed by an example of an cocartesian morphism lying over an arbitrary morphism in the base category.

\begin{definition}\label{opc}
A morphism $(\phi,g):(P,T,\J^\bullet) \to (P',T',\J'^\bullet)$ in $\Sh^{\C^\times}$ is cocartesian if, given any morphism $(\psi,h): (P,T,\J^\bullet)  \to (P'', T'',\J''^\bullet)$ in $\Sh^{\C^\times}$ and any morphism $\nu: (P'',T'') \to (P',T')$ in $\Orb$ such that $\psi = \phi \circ \nu$, there exists a unique morphism $(\nu,f):(P',T',\J'^\bullet)  \to (P'', T'',\J''^\bullet)$ such that 
\[
\begin{tikzcd}
\ar[rd,"{(\psi,h)}"'] (P,T,\J^\bullet) \ar[r,"{(\phi,g)}"]& (P', T',\J'^\bullet) \ar[d,"{(\nu,f)}"] \\
& (P'',T'',\J''^\bullet)  
\end{tikzcd}
\]
commutes. 
\end{definition}

\begin{example}\label{exo}
Let $\phi: (P,T) \to (P',T')$ be a morphism in $\Orb$ and $\J^\bullet$ be an object in $\Sh(T'^\C,\O^\bullet)^{\C^\times}$. Then $(\phi,\id_{\phi^*\F}): (P',T',\J^\bullet) \to (P,T,\phi^*\J^\bullet)$ is a cocartesian morphism in $\Sh^{\C^\times}$.
\end{example}

\begin{definition}
A subfunctor of $\Hom(\C,-)$ is a functor $S: \Orb \to \CxMfd^{\C^\times}$ sending $(P,T)$ to a $\C^\times$-equivariant open submanifold $S(T) := S(P,T)$ of $T^\C := \Hom(\C,T)$. It sends a morphism $\phi:(P,T) \to (P',T')$ to a morphism of $\C^\times$-equivariant manifolds $S(\phi):S(T) \to S(T')$ such that the diagram
\begin{equation}\label{sdef}
\begin{tikzcd}
S(T) \ar[rr,"S(\phi)"] \ar[d,hook] &&S(T') \ar[d,hook] \\
T^\C \ar[rr,"\phi"] && T'^{\C}
\end{tikzcd}
\end{equation}
commutes in $\CxMfd^{\C^\times}$.
\end{definition}

For a subfunctor $S$ of $\Hom(\C,-)$, there is a category $\Sh_S^{\C^\times}$ whose definition is exactly analogous to the definition of $\Sh^{\C^\times}$, only replacing $T^\C$ with $S(T)$ for each $(P,T)$ in $\Orb$. It is easy to see that the forgetful functor $\Sh_S^{\C^\times} \to \Orbop$ is also a cocartesian fibration by considering morphisms of the kind in Example \ref{exo} above. Furthermore, for each object $(P,T) \in \Orb$, there is a functor  
\[
\Res_{S(T)}: \Sh(T^\C,\O^\bullet)^{\C^\times} \to  \Sh(S(T),\O^\bullet)^{\C^\times}
\]
which sends a sheaf $\J^\bullet$ to its restriction $\J^\bullet_{S(T)}$ along the inclusion $S(T) \subset T^\C$. It may be verified using the diagram \ref{sdef} that 
\[
(\phi^*\J^\bullet)_{S(T)} \cong S(\phi)^*(\J^\bullet_{S(T')})
\]
for each morphism $\phi:(P,T) \to (P',T')$ in $\Orbop$. It follows that the functors $\Res_{S(T)}$ assemble over $\Orb$ to yield a functor $\Res_S$ such that the diagram
\[
\begin{tikzcd}
\Sh^{\C^\times} \ar[rr,"\Res_S"] \ar[rd] & & \Sh_S^{\C^\times} \ar[dl] \\
&\Orbop& 
\end{tikzcd}
\]
commutes.

\begin{example}
The functor $\Fr: \mathrm{Orb}^{\ab}_{\L^2_G} \to \CxMfd^{\C^\times}$ of Definition \ref{subb} is a subfunctor of $\Hom(\C,-)$ and thus determines a restriction functor 
\[
\Res_{\Fr}: \Sh^{\C^\times} \longrightarrow \Sh_{\Fr}^{\C^\times}.
\]
of categories over $\mathrm{Orb}^{\ab,\op}_{\L^2_G}$. 
\end{example}

\section{Complex analytic equivariant cohomology}\label{s5}

Let $T$ be a compact abelian group. In this section we give the definition of complex analytic equivariant cohomology $\H^\bullet_T$ and prove several of its fundamental properties.\footnote{The name {\em complex analytic equivariant cohomology} is due to Rezk \cite{Rezk}, who adapted it from Definition 6.2 in \cite{Spong} to this context.}

\subsection{Definition of $\H^\bullet_T$} 

For a finite $T$-complexes $X$, we begin by considering the graded tensor product of $\O^\bullet_T$ with the Borel $T$-equivariant cohomology ring $H^\bullet_T(X;\Z)$, regarded as a constant sheaf over $T^\C$. The tensor product is to be defined over a graded ring homomorphism that we now describe.

The Borel-equivariant cohomology $H^\bullet(X \times_T ET;\Z)$ of a $T$-space $X$ has the structure of a graded commutative algebra over $H^\bullet(BT;\Z)$ induced by the canonical map $X \times_T ET \to BT$. It is well known that $H^\bullet(BT;\Z)$ is isomorphic to the polynomial ring freely generated by elements $t_1,...,t_d$ in degree two, where $d$ is the rank of $T$. Consider the composite map
\[
H^2(BT;\Z) \cong \Hom(T,\R/\Z) \xrightarrow{\Hom(\C,-)} \Hom_\C(T^\C,(\R/\Z)^\C) \hookrightarrow \O_T^2(T^\C) 
\]
where the isomorphism is the inverse of the map sending a homomorphism $\phi:T \to \R/\Z$ to the first Chern class of the $\R/\Z$-bundle associated to $\phi$. The composite sends $t \in H^2(BT;\Z)$ to the $\C$-linear map $T^\C \to \omega := (\R/\Z)^\C$ induced by the Lie group homomorphism $T \to \R/\Z$ corresponding to $t$. Since $H^\bullet(BT;\Z)$ is freely generated by degree two elements, the composite map determines a homomorphism of graded rings 
\begin{equation}\label{ringmap}
H^\bullet(BT;\Z) \longrightarrow \O^\bullet_T(T^\C)^{\C^\times}.
\end{equation}
Recall the $\C^\times$-action on $\O^\bullet_T(T^\C)$ given by $(\lambda \cdot f)(\gamma) = \lambda^{-j}f(\lambda \cdot \gamma)$ for a holomorphic function $f: T^\C \to \omega^{\otimes j}$. The homomorphism \eqref{ringmap} takes values in the graded subring of $\C^\times$-invariants in $\O^\bullet_T(T^\C)$, because the images of the generators $t \in H^2(BT;\Z)$ are $\C^\times$-invariant (as they are $\C$-linear). We equip $H^\bullet(BT;\Z)$ and $H^\bullet(X \times_T ET;\Z)$ with the trivial $\C^\times$-action, so that the homomorphism \eqref{ringmap} is $\C^\times$-equivariant. From now on, we write $\underline{H}^\bullet(X)$ for the constant sheaf over $T^\C$ with value $H^\bullet(X)$. Thus $\O_T^\bullet$ and $\underline{H}^\bullet_T(X) := \underline{H}^\bullet(X \times_T ET;\Z)$ are both $\C^\times$-equivariant sheaves of graded commutative $\underline{H}^\bullet_T(\pt)$-algebras. We can now give the following definition/proposition.

\begin{proposition}
Let $X$ be a finite $T$-complex. There is a functor 
\[
\F^\bullet_T: \Fin{T}^{\op} \longrightarrow \Sh(T^\C,\O^\bullet)^{\C^\times}
\]
sending $X$ to the tensor product
\[
\F^\bullet_T(X) := \underline{H}_T^\bullet(X) \otimes_{\underline{H}^\bullet_T(\pt)} \O^\bullet_{T}
\]
of $\C^\times$-equivariant graded commutative $\underline{H}^\bullet_T(\pt)$-algebras. The functor is exact, homotopy-invariant and is equipped with the suspension isomorphism induced by Borel equivariant cohomology.
\end{proposition}

\begin{proof}
The functor sends a map $f: X \to Y$ of finite $T$-complexes to the map $f^* \otimes_{\underline{H}^\bullet_T(\pt)} \id$ of  $\C^\times$-equivariant $\O^\bullet_T$-algebras. It is clear that this preserves identity maps and composition. The suspension isomorphism and the properties of exactness and homotopy-invariance are inherited from Borel cohomology, where it follows from Proposition 2.8 in \cite{Rosu} that the ring homomorphism \eqref{ringmap} is flat.
\end{proof}

\begin{remark}
If $C$ is a complete category, then right Kan extension along the inclusion 
\[
\iota_T: \Fin{T}^{\op} \hookrightarrow \Topv{T}^{\op}
\]
of the category of finite $T$-complexes into the category of all $T$-spaces defines a functor
\[
\Ran_{\iota_T}: \Fun(\Fin{T}^{\op}, C) \to \Fun(\Topv{T}^{\op},C).
\]
The value of $\Ran_{\iota_T}(F)$ on a $T$-space $X$ is the limit of the functor
\[
(\Fin{T}_{/X})^{\op} \xrightarrow{forget_X} (\Fin{T})^{\op} \xrightarrow{F} C.
\]
Here $\Fin{T}_{/X}$ denotes the category of pairs $(Y,\alpha)$, where $Y$ is a finite $T$-complex and $\alpha:Y \to X$ is a $T$-equivariant map. A morphism $f: (Y,\alpha) \to (Y',\alpha')$ is a $T$-cellular map $f:Y \to Y'$ such that $\alpha' \circ f = \alpha$. The functor $forget_X$ forgets the map to $X$. The value of $\Ran_{\iota_T}(F)$ on $X$ is given by
\[
\Ran_{\iota_T}(F)(X) = \lim_{\iota_T(Y) \xrightarrow{\alpha} X} F(Y)
\]
where the limit runs over all $(Y,\alpha)$. It will not hurt to drop the reference to $\iota_T$ from this notation in the sequel. The value of $\Ran_{\iota_T}(F)(f)$ on a $T$-map $f:X \to Y$ is the map of limits induced by pullback along the functor
\[
f_*:\Fin{T}_{/X} \to \Fin{T}_{/Y}
\]
given by postcomposition with $f$. 
\end{remark}

\begin{definition}\label{Hdef}
Let $\H^\bullet_T$ be the right (pointwise) Kan extension 
\[
\H^\bullet_T := \Ran_{\iota_T} \F^\bullet_T
\]
of $\F^\bullet_T$ along $\iota_T: \Fin{T}^{\op} \to \Topv{T}^{\op}$. Its value on a $T$-space $X$ is equal to the inverse limit
\[
\H^\bullet_T(X) = \lim_{Y \xrightarrow{\alpha}  X} \underline{H}_T^\bullet(Y_\alpha) \otimes_{\underline{H}^\bullet(BT)} \O^\bullet_T
\]
over all $T$-maps $Y \to X$ from a finite $T$-complex $Y$.
\end{definition}

\begin{remark}
If $X$ is a finite $T$-complex, then $\id_X: X \to X$ is an initial object in $\Fin{T}_{/X}^{\op}$. Therefore, we may identify the functor $\H_T^\bullet \circ \iota_T$ with the functor $\F^\bullet_T$.
\end{remark}

\begin{remark}\label{limopen}
If $S$ is an open $\C^\times$-equivariant submanifold of $T^\C$, then the restriction-to-$S$ functor $\Res_S$ commutes with the limit in Definition \ref{Hdef}. In other words, if $S$ is open, then $\Res_S$ preserves right Kan extension along $\iota_T$:
\[
\Res_S \circ (\Ran_{\iota_T} (\F^\bullet_T)) = \Ran_{\iota_T} (\Res_S \circ \F^\bullet_T).
\]
\end{remark} 

\subsection{Localisation theorem for $\H^\bullet_T$}

In this section, we adapt the proof of Theorem 7.2 in \cite{Spong} to the more general context of this paper. Fix a compact abelian group $T$, an arbitrary $T$-space $X$, and let $S$ be a subset of $T^\C = \Hom(\C,T)$. We write $X^S$ for the union
\[
X^S := \bigcup_{\gamma \in S} X^\gamma  \quad \subset X
\]
where $X^\gamma$ denotes the subspace $X^{\widebar{\gamma(\C)}} \subset X$ of points fixed by the closure of $\gamma(\C)$ in $T$. For any $\gamma \in \Hom(\C,T)$ and any two subsets $S,U$ of $T^\C$, we have that:
\begin{enumerate}[(i)]
\item $X^\gamma = X^{\lambda \cdot \gamma}$ for all $\lambda \in \C^\times$,
\item $(X^S)^U = (X^U)^S = X^U \cap X^S$,
\item $X^{S \cup U} = X^S \cup X^U$, 
\item $(X\times Y)^S \subseteq X^S \times Y^S$, with equality when $S = \{\gamma\}$,
\item $X^S = X$ whenever $S$ contains the trivial homomorphism.
\end{enumerate}
The first property holds because $\lambda$ acts on a homomorphism $\gamma: \C \to T$ by composing with the scalar action $\C \xrightarrow{\bar{\lambda}} \C$, and such an action preserves the image of $\gamma$. The other properties are straightforward to verify. 

Note that the mapping $X \mapsto X^S$ defines functors 
\[
(-)^S: \Topv{T} \to \Topv{T} \qquad \text{and} \qquad (-)^S: \Fin{T} \to \Fin{T}.
\]
In the following lemma, we consider the natural transformation 
\begin{equation}\label{restrict}
\begin{tikzcd}
& \ar[dd,Rightarrow,shorten=8mm,"i^{\op}"] & \\
 \Fin{T}^{\op}   \ar[rr,bend left=5mm,"\id" {xshift=0mm} ] \ar[rr,bend right=5mm,"(-)^S"'{xshift=0mm}] & & \Fin{T}^{\op}  \ar[rr,"\F^\bullet_T"]& & \Sh(T^\C,\O^\bullet)^{\C^\times}\\
& \ar[phantom] & 
\end{tikzcd}
\end{equation}
induced by the inclusion $i_X: X^S \hookrightarrow X$ for each finite $T$-complex $X$.

\begin{lemma}\label{finloc}
Let $S$ be a subset of $T^\C$ and $X$ a finite $T$-complex. The inclusion $i_X: X^S \hookrightarrow X$ induces an isomorphism 
\[
\F^\bullet_T(X)_S \xrightarrow{\sim} \F^\bullet_T(X^S)_S
\]
natural in $X$.
\end{lemma}

\begin{proof}
Let $\gamma$ be an arbitrary homomorphism $\C \to T$.  It suffices to show that the inclusion $X^\gamma \subset X$ induces an isomorphism 
\[
\F^\bullet_T(X)_{\gamma} \xrightarrow{\sim}  \F^\bullet_T(X^\gamma)_{\gamma}
\]
of stalks at $\gamma$. Since $\F^\bullet_T$ is a cohomology theory on finite $T$-complexes, a Mayer-Vietoris argument along with an induction on the cells of $X$ allows us to reduce to the case where $X$ is a homogeneous space, say $T/K$ for a closed subgroup $K \leq T$. If $\widebar{\gamma(\C)} \leq K$, then the inclusion map $X^\gamma \subset X$ is the identity, which induces an identity on sheaves. So we may assume that $\widebar{\gamma(\C)} \not\leq K$, which reduces the argument to showing that 
\[
\F^\bullet_T(T/K)_{\gamma} \to \F^\bullet_T(\emptyset)_{\gamma} = 0
\]
is an isomorphism. In other words, it suffices to show that
\[
\F^\bullet_T(T/K)_{\gamma} = \underline{H}^\bullet_T(T/K) \otimes_{\underline{H}^\bullet(BT)} \O^\bullet_{T, \gamma} \cong \underline{H}^\bullet(B K) \otimes_{\underline{H}^\bullet(BT)} \O^\bullet_{T, \gamma} 
\]
is trivial. This follows if there exists a homomorphism in $\Hom(T,\bT) \cong H^2(BT)$ which is invertible in $\O^\bullet_{T, \gamma}$ and whose image in $H^2(B K) \cong \Hom(K,\bT)$ is zero (i.e. a homomorphism which is nontrivial on $\widebar{\gamma(\C)}$ but trivial on $K$). Since $\widebar{\gamma(\C)} \not\leq K$ and $K$ is closed, such a homomorphism must exist. This completes the proof.
\end{proof}

We also use $i$ to denote the natural transformation $(-)^S \hookrightarrow \id$ of functors from $\Topv{T}$ into $\Topv{T}$, analogously to Diagram \eqref{restrict}. 

\begin{theorem}\label{local}
Let $S$ be a $\C^\times$-equivariant open subset of $T^\C$. There is a natural isomorphism 
\[
(\Res_S \circ \H^\bullet_T)(i): \Res_S \circ \H^\bullet_T \xrightarrow{\sim}  \Res_S \circ \H^\bullet_T \circ (-)^S
\] 
whose component on a $T$-space $X$ is the map
\[
(\Res_S\circ \H^\bullet_T)(i_X): \H^\bullet_T(X)_S \xrightarrow{\sim} \H^\bullet_T(X^S)_S
\]
induced by the inclusion $i_X: X^S \hookrightarrow X$.
\end{theorem}

\begin{proof}
By Lemma \ref{finloc}, there is a natural isomorphism 
\[
(\Res_S \circ \F^\bullet_T)(i): \Res_S \circ \F_T^\bullet \xrightarrow{\sim}  \Res_S \circ \F_T^\bullet \circ (-)^S
\]
whose component on a finite $T$-complex $X$ is the map
\[
(\Res_S \circ \F^\bullet_T) (i_X): \F^\bullet_T(X)_S \xrightarrow{\sim} \F^\bullet_T(X^S)_S
\]
induced by $i_X: X^S \hookrightarrow X$. We apply the right Kan extension functor along 
\[
\iota_T: \Fin{T} \hookrightarrow \Topv{T}
\]
to obtain the natural isomorphism $\Ran_{\iota_T}((\Res_S\circ\F_T^\bullet)(i))$. Naturality of the latter with respect to the inclusion map $i_X: X^S \hookrightarrow X$ says that the diagram 
\[
\begin{tikzcd}
\Ran_{\iota_T} (\Res_S \circ \F^\bullet_T)(X) \ar[rrrrr,"\Ran_{\iota_T}(\Res_S \circ \F^\bullet_T)(i_X)"] \ar[d,"\Ran_{\iota_T}((\Res_S \circ \F^\bullet_T)(i))_{X}","\sim"'] &&&&& \Ran_{\iota_T} (\Res_S \circ \F^\bullet_T)(X^S) \ar[d,"\Ran_{\iota_T}((\Res_S \circ \F^\bullet_T)(i))_{X^S}","\sim"'] \\
\Ran_{\iota_T} (\Res_S \circ \F^\bullet_T \circ (-)^S)(X) \ar[rrrrr,"\Ran_{\iota_T}(\Res_S \circ \F^\bullet_T \circ (-)^S)(i_X)"] &&&&&\Ran_{\iota_T}(\Res_S \circ \F^\bullet_T \circ (-)^S)(X^S)
\end{tikzcd}
\]
commutes. Since $S$ is open, we have
\[
\Ran_{\iota_T}(\Res_S\circ \F^\bullet_T)(i_X) = (\Res_S \circ \Ran_{\iota_T}(\F^\bullet_T))(i_X) = (\Res_S \circ \H^\bullet_T)(i_X)
\]
where the first equality holds by Remark \ref{limopen} and the second by definition of $\H^\bullet_T$. Therefore, it suffices to show that the upper horizontal arrow is an isomorphism. Since the vertical arrows are isomorphisms, this is true if and only if the lower horizontal arrow is an isomorphism. The latter is the map of limits
\[
\lim_{Y \to X} \F^\bullet_T(Y^S)_S \to \lim_{Y \to X^S} \F^\bullet_T(Y^S)_S
\]
induced by pullback along the functor $i_{X,*}: \Fin{T}_{/X^S} \to \Fin{T}_{/X}$ which is, in turn, induced by the inclusion $i_X: X^S \hookrightarrow X$. To see that this is an isomorphism, note that the functors
\[
\Fin{T}_{/X} \xrightarrow{(-)^S} \Fin{T}_{/X^S} \xrightarrow{i_{X,*}} \Fin{T}_{/X}  \xrightarrow{ \F^\bullet_T(-^S)_S} \Sh(S)
\]
and  
\[
\Fin{T}_{/X}  \xrightarrow{ \F^\bullet_T(-^S)_S} \Sh(S)
\]
are equal. Therefore, the map of limits that is induced by pullback along $i_{X,*} \circ (-)^S$ is the identity. In the same way, we see that the pullback along $(-)^S \circ i_{X,*}$ also induces the identity on limits. The map of limits that is induced by pullback along $i_{X,*}$ is therefore an isomorphism. This completes the proof.
\end{proof}

\subsection{Properties of $\H^\bullet_T$}

In this subsection we use the localisation theorems to show that the restriction of $\H^\bullet_T(X)$ to a $\C^\times$-equivariant open submanifold $S \subset T^\C$ behaves well on a certain subcategory $\Topv{T}(S) \subset \Topv{T}$ depending on $S$. In particular, we show that it is a cohomology theory on this category. Furthermore, we show that there is a localisation theorem on stalks.

\begin{definition}\label{localfin}
Let $S \subset T^\C$ be a $\C^\times$-equivariant open submanifold and let $X$ be a $T$-space. We say that $X$ is {\em locally finite over $S$} if, for each $\gamma \in S$:
\begin{enumerate}[(i)]
\item there is a $\C^\times$-equivariant open neighbourhood $U \subset S$ of $\gamma$ such that $X^U = X^\gamma$ (as subspaces of $X$), and 
\item $X^\gamma$ has the $T$-homotopy type of a finite $T$-complex. 
\end{enumerate}
We write $\Topv{T}(S) \subset \Topv{T}$ for the full subcategory on spaces locally finite over $S$.
\end{definition}

\begin{proposition}\label{stalks}
Let $X$ be locally finite over $S$ and let $\gamma \in S$. The inclusion $X^\gamma \hookrightarrow X$ of the subspace of points fixed by $\widebar{\gamma(\C)}$ induces an isomorphism
\[
\H^\bullet_T(X)_\gamma \xrightarrow{\sim} \H^\bullet_T(X^\gamma)_\gamma 
\]
of stalks at $\gamma$.
\end{proposition}

\begin{proof}
Choose a $\C^\times$-equivariant open neighbourhood $U \subset S$ of $\gamma$ such that $X^U = X^\gamma$. We have 
\[
\H^\bullet_T(X)_\gamma \cong \H^\bullet_T(X^U)_\gamma = \H^\bullet_T(X^\gamma)_\gamma
\] 
where the first map is the isomorphism of Proposition \ref{local} after taking stalks at $\gamma \in S$.
\end{proof}

\begin{definition}
We write $\widetilde{\H}^\bullet_T$ for the reduced version of $\H^\bullet_T$, which is a functor from the category $\Topv{T}_*$ of pointed $T$-spaces into the category $\Sh(T^\C,\O^\bullet\text{-}\mathrm{mod})^{\C^\times}$ of $\C^\times$-equivariant $\O^\bullet_T$-modules. Its value on a pointed $T$-space $X$ is the kernel of the map $\H^\bullet_T(X) \to \H^\bullet_T(\pt)$ induced by the inclusion of the basepoint. 
\end{definition}

We may identify $\widetilde{\H}^\bullet_T$ with the right Kan extension of $\widetilde{\F}^\bullet_T$ along
\[
\iota_{T,*}: \Fin{T}_* \hookrightarrow \Topv{T}_*
\]
since, for a pointed $T$-space $X$, the right Kan extension is a limit over $\Fin{T}_{/X}$ and this operation commutes with taking the kernel.

\begin{notation}
Let $\Topv{T}(S)_* \subset \Topv{T}_*$ be the full subcategory on those spaces whose underlying unpointed space is locally finite over $S$.
\end{notation}

If $X \in \Topv{T}(S)_*$, then the map of Proposition \ref{stalks} induces a natural isomorphism
\[
\widetilde{\H}^\bullet_T(X)_\gamma \xrightarrow{\sim} \widetilde{\H}^\bullet_T(X^\gamma)_\gamma 
\]
of stalks at $\gamma \in S$. 

\begin{proposition}
The composite functor
\[
\Topv{T}(S)_* \hookrightarrow \Topv{T}_* \xrightarrow{\widetilde{\H}^\bullet_T} \Sh(T^\C,\O^\bullet\text{-}\mathrm{mod})^{\C^\times} \xrightarrow{\Res_S} \Sh(S,\O^\bullet\text{-}\mathrm{mod})^{\C^\times} 
\]
takes cofibre sequences to exact sequences.
\end{proposition}

\begin{proof}
Let $X \to Y \to Z$ be a cofibre sequence in $\Topv{T}(S)_*$ and let $\gamma \in S$. Consider the induced sequence 
\[
\widetilde{\H}^\bullet_T(Z)_S \to \widetilde{\H}^\bullet_T(Y)_S \to \widetilde{\H}^\bullet_T(X)_S
\]
in $\Sh(S,\O^\bullet\text{-}\mathrm{mod})^{\C^\times}$. By Proposition \ref{stalks} the induced sequence of stalks at $\gamma$ is isomorphic to the sequence
\[
\widetilde{\F}^\bullet_T(Z^\gamma)_\gamma \to \widetilde{\F}^\bullet_T(Y^\gamma)_\gamma \to \widetilde{\F}^\bullet_T(X^\gamma)_\gamma
\]
induced by the cofibre sequence $X^\gamma \to Y^\gamma \to Z^\gamma$. Since $\F^\bullet_T$ is a cohomology theory on finite $T$-complexes, the latter sequence is exact. This completes the proof.
\end{proof}

\begin{proposition}\label{suspension}
The suspension map of Borel cohomology induces a natural transformation
\[
\widetilde{\H}^{\bullet+1}_T(S^1 \wedge X) \rightarrow \widetilde{\H}^{\bullet}_T(X)
\]
of functors from $\Topv{T}_*$ into the category of $\C^{\times}$-equivariant $\O^\bullet_T$-modules. It is a natural isomorphism on the subcategory $\Topv{T}(S)$ after restriction to $S \subset T^\C$.
\end{proposition}

\begin{proof}
Consider the diagram
\[
\begin{tikzcd}
\Sh(T^\C,\O^\bullet\text{-}\mathrm{mod})^{\C^\times} \\
\\
\\
\Fin{T}_* \ar[rrr,"\iota_{T,*}"] \ar[uuu,"\widetilde{\F}^\bullet_T"{name=A}] &&& \Topv{T}_* \ar[llluuu,bend left=5mm,"\widetilde{\H}^{\bullet+1}_T  (S^1 \wedge -)"'{name=B}]  \ar[llluuu,bend right=10mm,"\widetilde{\H}^\bullet_T"'{name=C}]
\ar[Rightarrow,from=B,to=A,"\alpha",shorten=8mm,yshift=-2mm,"\sim"'] \ar[Rightarrow,from=A,to=C,"\beta",xshift=2mm,shorten=15mm]
\end{tikzcd}
\]
The arrow $\alpha$ is the natural transformation whose component on a finite $T$-complex $Y$ is the map
\[
\widetilde{\underline{H}}^{\bullet+1}_T(S^1\wedge Y) \otimes_{\underline{H}^\bullet(BT)} \O^\bullet_T  \to \widetilde{\underline{H}}^{\bullet}_T(Y) \otimes_{\underline{H}^\bullet(BT)} \O^\bullet_T
\]
induced by the natural suspension isomorphism of Borel $T$-equivariant cohomology. The arrow $\beta$ is the unique one making the diagram commute, which exists by the universal property of the Kan extension $\H^\bullet_T$ (recalling that $\tilde{\H}^\bullet_T  \circ \iota_{T,*}= \tilde{\F}^\bullet_T$). The arrow $\beta$ is the natural transformation whose existence is asserted in the proposition. It remains to show that $\beta$ is an isomorphism over $S \subset T^\C$. Let $X$ be an arbitrary $T$-space. Locally at $\gamma \in S$, the component of the arrow $\beta$ at $X$ has as its source 
\[
\widetilde{\H}^{\bullet + 1}_T(S^1\wedge X)_\gamma \cong \widetilde{\H}^{\bullet + 1}_T((S^1\wedge X)^\gamma)_\gamma \cong  \widetilde{\H}^{\bullet + 1}_T(S^1\wedge X^\gamma)_\gamma
\]
where we have used Proposition \ref{stalks} and the fact that $(S^1\wedge X)^\gamma  = S^1\wedge X^\gamma$. Since $X^\gamma$ is finite up to $T$-homotopy, on this space $\beta$ is determined by its restriction along $\iota_{T,*}$, where it is equal to $\alpha$, which is an isomorphism. This completes the proof.
\end{proof}

\subsection{Examples and computations}

\begin{example}
If $X$ is a space with the action of the trivial group $T = e$, then $\H^\bullet_T(X)$ is the sheaf over $T^\C = \{0\}$ with stalk given by the inverse limit 
\[
\lim_{Y_\alpha \to X} H^\bullet_T(Y_\alpha) \otimes_{\Z} \C^\bullet
\]
over all $T$-maps $Y_\alpha \rightarrow{} X$ from finite CW-complexes $Y_\alpha$ into $X$. Here $\C^\bullet = \oplus_{n \in \Z} \C$ is a $\Z$-algebra via the ring map given by the inclusion of $\Z$ into the degree zero part of $\C^\bullet$.
\end{example}

\begin{example}
Let $X$ be a space with trivial $T$-action. Then any map $Y \rightarrow{} X$ from a finite $T$-complex $Y$ into $X$ factors through the quotient $Y \to Y/T \to X$, which is a finite CW-complex. Therefore, $\H^\bullet_T(X)$ may be identified with the inverse limit  
\[
 \lim_{Y_\alpha \to X}\underline{H}^\bullet(Y_\alpha) \otimes_{\underline{\Z}} \O^\bullet_T
\]
over all maps $Y_\alpha \rightarrow{} X$ from finite (non-equivariant) CW-complexes $Y_\alpha$ into $X$. 
\end{example}

\begin{example} 
Let $X \to X/T$ be a principal $T$-bundle. Since the $T$-action is free, we have
\[
X^\gamma = \begin{cases} X &\text{if  } \gamma = 0 \\ \emptyset &\text{else.  } \end{cases}
\]
So, by applying Proposition \ref{local} with $S = T^\C \setminus \{0\}$ we see that $\H_T^\bullet(X)$ is supported over $0 \in T^\C$. We also have
\begin{align*}
\H^\bullet_T(X) &=  \lim_{Y_\alpha \to X} \underline{H}_T^\bullet(Y_\alpha) \otimes_{\underline{H}^\bullet(BT)} \O^\bullet_T\\
&\cong \lim_{Y_\alpha \to X} \underline{H}^\bullet(Y_\alpha/T) \otimes_{\underline{H}^\bullet(BT)} \O^\bullet_T\\
&\cong \lim_{Z_\alpha \to X/T} \underline{H}^\bullet(Z_\alpha) \otimes_{\underline{H}^\bullet(BT)} \O^\bullet_T.
\end{align*}
In the second line we have used the homotopy equivalence $Y \times_T ET \simeq Y/T$ for a free finite $T$-complex $Y$, and the $H^\bullet(BT)$-algebra structure is induced by the map $Y/T \to BT$ classifying $Y \to Y/T$ as a principal $T$-bundle. In the third line the $H^\bullet(BT)$-algebra structure is induced by the composite $Z \to X/T \to BT$ for a finite, non-equivariant CW-complex $Z$ over $X/T$, where the second map classifies the principal $T$-bundle $X \to X/T$. 
\end{example}

The following calculation shows that $\H_T^\bullet$ is not a Borel-equivariant cohomology theory. 

\begin{proposition}
The sheaf $\H_{S^1}^\bullet(E S^1)$ is supported at $0 \in (S^1)^\C = \C$, where its stalk is the formal completion 
\[
\H_{S^1}^\bullet(E S^1)_0 \cong \H^\bullet_{S^1}(\pt)^{\wedge}_0
\]
of the ring $\H^\bullet_{S^1}(\pt)_0 = \O^\bullet_{S^1,0}$ of holomorphic germs at $0$.
\end{proposition}

\begin{proof}
Since $ES^1$ is a free $S^1$-space, it has support at $0 \in (S^1)^\C = \C$.  By the previous example, we have
\[
\H_{S^1}^\bullet(E S^1)_0 \cong \left( \lim_{Y_\alpha \to BS^1} H^\bullet(Y_\alpha) \otimes_{H^\bullet(BS^1)}  \O_{S^1}^\bullet \right)_0.
\]
The inverse limit may be computed over the subdiagram of finite CW-complexes
\[
\C P^1 \hookrightarrow \C P^2 \hookrightarrow  ... \subset \C P^{\infty}  = BS^1.
\]
We have
\begin{align*}
\left(\lim_{n} H^\bullet(\C P^n) \otimes_{H^\bullet(\C P^\infty)}  \O_{S^1}^\bullet \right)_0 &\cong \left(\lim_{n} \Z[t]/t^{n+1} \otimes_{\Z[t]} \O^\bullet_{S^1}\right)_0 \\
&\cong (\lim_{n} \O^\bullet_{S^1}/t^{n+1})_0 \\
&\cong \ \O^{\bullet,\wedge}_{S^1,0} \ = \ \H^\bullet_{S^1}(\pt)^{\wedge}_0.
\end{align*}
\end{proof}

\begin{proposition}\label{orbit}
Let $j: \Hom(\C,K) \hookrightarrow \Hom(\C,T)$ be the map induced by the inclusion of a closed subgroup $K \leq T$. There is an isomorphism
\[
j^* \H^\bullet_T(T/K) \cong \O^\bullet_K
\]
in $\Sh(K^\C,\O^\bullet)^{\C^\times}$.
\end{proposition}

\begin{proof}
We have isomorphisms
\[
j^* \H^\bullet_T(T/K) \cong j^{-1} \underline{H}^\bullet_T(T/K) \otimes_{j^{-1}\underline{H}^\bullet(BT)} j^* \O^\bullet_T \cong  \underline{H}^\bullet(BK) \otimes_{\underline{H}^\bullet(BK)} j^* \O^\bullet_T \cong j^* \O^\bullet_T \cong \O^\bullet_K
\]
in $\Sh(K^\C,\O^\bullet)^{\C^\times}$. The first holds by definition and because $T/K$ is a finite $T$-complex, the second uses that $T/K \times_T ET$ is a model for $BK$, and the third and fourth are the canonical isomorphisms.
\end{proof}

\section{Equivariant elliptic cohomology}\label{s6}

The goal of this section is to give the definition of $G$-equivariant elliptic cohomology $\Ell_G^\bullet$ and to prove some of its basic properties. To this end, let $\G$ be a groupoid enriched over $\Top$ and consider $\Top$ as a category enriched over itself. Let $\Topv{\G}$ be the category $\Fun(\G^{\op},\Top)$ of enriched functors with natural transformations as morphisms. We will think of this as the category of ``$\G$-spaces". The main ingredient in this section is a certain functor 
\[
\H^{\bullet}_\G: \Topv{\G}^{\op} \times \Orbop \longrightarrow \Sh^{\C^\times}.
\]
We produce $\Ell_G^\bullet$ from this functor essentially by setting $\G = \L^2_G$. 

\subsection{Functoriality of $\H^\bullet_\G$ in $\Orb$}

We will construct $\H^{\bullet}_\G$ by regarding it as a functor
\[
\H^{\bullet}_\G: \Orbop \longrightarrow \Fun( \Topv{\G}^{\op},\Sh^{\C^\times}).
\]
In fact, we construct a family of such functors $^S\H^{\bullet}_\G$, indexed over all subfunctors $S$ of $\Hom(\C,-)$. This is necessary in order to formulate certain localisation properties of $\Ell_G^\bullet$. Then we will set 
\[
\H^{\bullet}_\G := \ ^{\Hom(\C,-)} \H^{\bullet}_\G.
\]

\begin{remark}\label{inddef}
For each homomorphism $\phi:T \to T'$ of compact abelian groups there is a functor
\[
\ind_\phi: \Topv{T} \to \Topv{T'} \quad \text{where} \quad \ind_\phi Y := Y \times_\phi T' := Y \times T'/(yt,t') \sim (y,\phi(t)t').
\]

If $Y$ has a (finite) $T$-complex structure, then there is a canonical (finite) $T'$-complex structure on $\ind_\phi Y$. In fact, $\ind_\phi$ restricts to a functor $\ind_\phi: \Fin{T} \to \Fin{T'}$.
\end{remark}

\begin{definition}
Let $S$ be a subfunctor of $\Hom(\C,-)$ and let $(P,T)$ be an object in $\Orb$. We write $ev^S_{P,T}$ for the functor 
\[
ev^{S}_{P,T}:\Topv{\G} \to \Topv{T}
\]
sending a topological functor $F$ to the $T$-space $F(P)^{S(T)}$ and a natural transformation $f:F \to F'$ to the $T$-equivariant map 
\[
ev^{S}_{P,T}(f): F(P)^{S(T)} \to F'(P)^{S(T)}
\]
induced by the component $f_P: F(P) \to F'(P)$ of $f$ at $P$. 
\end{definition}

\begin{proposition}\label{functor}
Let $S$ be a subfunctor of $\Hom(\C,-)$. There is a functor  
\[
^S\H^{\bullet}_\G: \Orbop \longrightarrow \Fun(\Topv{\G}^{\op}, \Sh^{\C^\times})
\]
sending an object $(P,T)$ to the composite functor 
\[
\Topv{\G}^{\op} \xrightarrow{ev_{P,T}^{S,\op}} \Topv{T}^{\op} \xrightarrow{\H^\bullet_T} \Sh(T^\C,\O^\bullet)^{\C^\times} \hookrightarrow \Sh^{\C^\times}
\]
which sends $F$ to $(P,T,\H^\bullet_T(F(P)^{S(T)}))$. 
\end{proposition}

\begin{proof}

Let $\phi$ be a morphism $(P,T) \to (P',T')$ in $\Orb$. We define $^S\H^\bullet_\G(\phi)$ to be the pair $(\phi,h^S_\G(\phi))$, where $h^S_\G(\phi)$ is the composite of the natural transformations displayed in the diagram 
\begin{equation}\label{nat1}
\begin{tikzcd}
\Topv{\G}^{\op} \ar[rr,"ev^{S,\op}_{P,T}"] \ar[drr,"ev^{S,\op}_{P',T'}"'{name=c},bend right=5mm] \ar[Rightarrow,from=c,to=Z,shorten=3mm,blueish,"\mu^{S,\op}_g"]&&|[alias=Z]| \Topv{T}^{\op} \ar[rr,"\H^\bullet_T"{name=a}]\ar[d,"\ind_\phi"']&&|[alias=Y]| \Sh(T^\C) \ar[rr,hook] && \Sh^{\C^\times} \\
&& \Topv{T'}^{\op} \ar[rr,"\H^\bullet_{T'}"{name=b}] && \Sh(T'^\C)\ar[u,"\phi^*"] \ar[urr,hook,bend right=5mm,""{name=d}] & \ar[Rightarrow,from=b,to=a,shorten=2mm,xshift=0mm,yshift=-1mm,blueish,"\xi_\phi"]\ar[Rightarrow,blueish,lu,shorten=5mm,yshift=2mm,""']&
\end{tikzcd}
\end{equation}
The component of $h^S_\G(\phi)$ on a given $F \in \Topv{\G}$ is thus a composite of maps
\[
(\phi^*\circ \H^\bullet_{T'})(\mu^{S,\op}_{g,F}): \phi^*\H^\bullet_{T'}(F(P')^{S(T')}) \to \phi^*\H^\bullet_{T'}(\ind_\phi (F(P)^{S(T)}))
\]
and 
\[
\xi_{\phi,F(P)^{S(T)}}: \phi^*\H^\bullet_{T'}(\ind_\phi (F(P)^{S(T)})) \to \H^\bullet_T(F(P)^{S(T)}).
\]
We may ignore the natural transformation on the right of Diagram \eqref{nat1}, since the sheaf component of this natural transformation is just the identity. Namely, it sends a sheaf $\J^\bullet$ in $\Sh(T'^\C)$ to the morphism 
\[
(\phi,\id_{\phi^*\J^\bullet}): (P',T',\J^\bullet) \to (P,T,\phi^*\J^\bullet)
\]
in $\Sh^{\C^\times}$. 

We begin by describing the natural transformation $\mu^S_g$, which depends on a choice of morphism $g:P' \to P$ in $\G$ representing $\phi$. Given a topological functor $F: \G^{\op} \to \Top$, the map $F(g):F(P) \to F(P')$ induces a $\phi$-equivariant map  
\[
F(g)^{S(\phi)}:F(P)^{S(T)} \to F(P')^{S(T')}
\]
of fixed point subspaces. The component of $\mu^S_g$ on $F$ is defined as the canonical $T'$-map 
\[
\mu^S_{g,F}:\ind_\phi (F(P)^{S(T)}) \to F(P')^{S(T')}
\]
associated to $F(g)^{S(\phi)}$. To see that $\mu^S_{g,F}$ is natural in $F$, we observe that the diagram
\[
\begin{tikzcd}
\ind_\phi (F(P)^{S(T)}) \ar[r,"\mu^S_{g,F}"] \ar[d,"\ind_\phi (ev^S_{P,T}(f))"'] & F(P')^{S(T')} \ar[d,"ev^S_{P',T'}(f)"] \\ 
\ind_\phi (F'(P)^{S(T)}) \ar[r,"\mu^S_{g,F'}"] & F'(P')^{S(T')} 
\end{tikzcd}
\]
in $\Topv{T'}$ commutes for each morphism $f: F \to F'$ in $\Topv{\G}$, by the naturality of $f_P$ in $P$. Note that $\H^\bullet_{T'}(\mu^S_g)$ does not depend on the choice of $g$ representing $\phi$ for the following reason. If $g$ and $g'$ both represent $\phi$, then there is $t\in T'$ such that $g' = gt$, so that $g$ and $g'$ induce the same map after taking the homotopy quotient by $T'$ (multiplication by $t$ is homotopic to the identity in $ET'$).

The natural transformation $\xi_\phi$ is defined in the following way. Consider the diagram
\[
\begin{tikzcd}
\Sh(T^\C) && \ar[ll,"\phi^*"'] \Sh(T'^\C) & \\
\\
\Fin{T}^{\op} \ar[uu,color=blueish,"\F^\bullet_T"{name=C}] \ar[rr,color=blueish,"\ind_{\phi}"'xshift=2mm] \ar[rd,hook,"\iota_T"'] && \Fin{T'}^{\op}  \ar[rd,hook,"\iota_{T'}"'] \ar[uu,color=blueish,"\F^\bullet_{T'}"{name=D}]& \\
& \Topv{T}^{\op} \ar[luuu,crossing over,color=red,"\H^\bullet_T"'{name=a}] \ar[rr,color=red,"\ind_{\phi}"']&& \Topv{T'}^{\op}  \ar[luuu,color=red,"\H^\bullet_{T'}"'{name=b}]
\ar[Rightarrow,from=b,to=a,color=red,shorten=15mm,"\xi_\phi"',xshift=-5mm] \ar[Rightarrow,from=D,to=C,color=blueish,shorten=15mm,"\nu_\phi"',yshift=2mm,"\sim"]
\end{tikzcd}
\]
The component of the natural transformation 
\[
\nu_\phi: \phi^* \circ \F^\bullet_{T'} \circ \ind_\phi \to \F^\bullet_T
\]
on a finite $T$-complex $X$ is the map 
\begin{align*}
\nu_{\phi,X}: \phi^*\F^\bullet_{T'}(\ind_\phi X ) = \phi^{-1}\underline{H}^\bullet_{T'}(\ind_\phi X) \otimes_{\phi^{-1}\underline{H}^\bullet(BT')} \phi^*\O^\bullet_{T'} \xrightarrow{\sim} \underline{H}^\bullet_T(X) \otimes_{\underline{H}^\bullet(BT)} \O^\bullet_{T} = \F^\bullet_T(X)
\end{align*}
induced by the canonical equivalence $X_{hT} \xrightarrow{\sim}  (\ind_\phi X)_{hT'}$ over $BT'$ and the canonical isomorphism $\phi^* \O^\bullet_{T'} \xrightarrow{\sim} \O^\bullet_{T}$. The latter two maps induce a map on the tensor product since, when $X$ is a point, $BT \to BT'$ induces a map $\phi^{-1}\underline{H}^\bullet(BT') \to \underline{H}^\bullet(BT)$ which extends to the canonical map $\phi^* \O^\bullet_{T'} \xrightarrow{\sim} \O^\bullet_{T}$ of sheaves. Thus $\nu_\phi$ is clearly an isomorphism and natural in $X$.

Note that all squares and triangles in the diagram commute, except for the two squares containing $\nu_\phi$ and $\xi_\phi$ respectively. The source of the natural isomorphism $\nu_\phi$ is therefore 
\[
\phi^* \circ \F^\bullet_{T'} \circ \ind_\phi  = \phi^* \circ \H^\bullet_{T'} \circ \iota_{T'} \circ \ind_\phi = \phi^* \circ \H^\bullet_{T'} \circ \ind_\phi \circ \iota_{T}. 
\]
Recall that $\H^\bullet_T$ is the right Kan extension of $\F^\bullet_T$ along $\iota_T$, and that $\H^\bullet_T \circ \iota_T = \F^\bullet_T$. By the universal property of right Kan extension of , there is a unique natural transformation 
\[
\chi: \phi^* \circ \H^\bullet_{T'} \circ \ind_\phi \to \H^\bullet_T
\]
such that the left whiskering $\chi \iota_T$ by $\iota_T$ is equal to the natural transformation $\nu_\phi$. We set $\xi_\phi$ to be this natural transformation $\chi$.

We now check that composition is preserved. Let 
\[
\phi:(P,T) \to (P',T') \quad \text{and} \quad \phi': (P',T') \to (P'',T'')
\]
be morphisms in $\Orb$. By the rule for composition in the category $\Sh^{\C^\times}$, the composite of $^S\H^\bullet_\G(\phi)$ with $^S\H^\bullet_\G(\phi')$ is the pair consisting of $\phi' \circ \phi$ and the composite of the natural transformations in the diagram

\begin{equation}\label{nat3}
\begin{tikzcd}
&& |[alias=Z]| \Topv{T}^{\op} \ar[rr,"\H^\bullet_T"{name=c}] \ar[d,"\ind_\phi"] && \Sh(T^\C)  \\
\Topv{\G}^{\op} \ar[urr,"ev^{S,\op}_{P,T}",bend left=5mm] \ar[drr,"ev^{S,\op}_{P'',T''}"'{name=g},bend right=5mm] \ar[rr,"ev^{S,\op}_{P',T'}"{name=b}]  &&|[alias=Y]| \Topv{T'}^{\op} \ar[d,"\ind_{\phi'}"] \ar[rr,"\H^\bullet_{T'}"{name=e}]&&\Sh(T'^\C)  \ar[u,"\phi^*"']  \\
&& \Topv{T''}^{\op} \ar[rr,"\H^\bullet_{T''}"{name=f}]  && \Sh(T''^\C)\ar[u,"\phi'^*"']  \ar[uu,bend right=30mm,"(\phi'\circ \phi)^*"'{name=a}] \ar[Rightarrow,blueish, from=a, to=u,shorten=2mm,xshift=-1mm,"\sim"']  \ar[Rightarrow,blueish,from=f,to=e,shorten=2mm,"\xi_{\phi'}"'{yshift=0
mm},xshift=0mm]  \ar[Rightarrow,blueish,from=e,to=c,shorten=2mm,"\xi_\phi"'{yshift=0mm},xshift=0mm] \ar[Rightarrow,blueish,from=b,to=Z,shorten=3mm,"\mu^{S,\op}_g",yshift=-1ex]\ar[Rightarrow,blueish,from=g,to=Y,shorten=5mm,"\mu^{S,\op}_{g'}",xshift=1mm,yshift=-1mm]
\end{tikzcd}
\end{equation}
We treat this diagram in two parts. The first part is
\[
\begin{tikzcd}
&& |[alias=Z]|  \Topv{T}^{\op} \ar[d,"\ind_\phi"]  \ar[dd,bend left=30mm,"\ind_{\phi'\circ \phi}"{name=a}]  \\
\Topv{\G}^{\op} \ar[urr,"ev^{S,\op}_{P,T}",bend left=5mm] \ar[drr,"ev^{S,\op}_{P'',T''}"'{name=g},bend right=5mm] \ar[rr,"ev^{S,\op}_{P',T'}"{name=b}] && |[alias=X]| \Topv{T'}^{\op} \ar[d,"\ind_{\phi'}"] \\
&& \Topv{T''}^{\op} \ar[Rightarrow,blueish, from=X, to=a,shorten=2mm,"\sim"] 
\ar[Rightarrow,blueish,from=b,to=Z,shorten=2mm,"\mu^{S,\op}_g",yshift=-1ex]\ar[Rightarrow,blueish,from=g,to=X,shorten=5mm,"\mu^{S,\op}_{g'}",xshift=1mm,yshift=-1mm]
\end{tikzcd}
\]
We need to show that the composite of these natural transformations is equal to 
\[
\mu_{g\circ g'}^{S,\op}: ev^{S,\op}_{P'',T''} \longrightarrow \ind_{\phi' \circ \phi} \circ ev^{S,\op}_{P,T}.
\] 
Let $F \in \Topv{\G}^{\op}$. It suffices to show that the diagram
\[
\begin{tikzcd}
\ind_{\phi'\circ \phi} (F(P)^{S(T)})  & \ar[l,"\sim"] \ind_{\phi'}(\ind_\phi (F(P)^{S(T)}) \\
F(P'')^{S(T'')} \ar[u,"\mu_{g\circ g'}^{S,\op}"] \ar[r,"\mu_{g'}^{S,\op}"]& \ind_{\phi'} (F(P')^{S(T')})\ar[u,"\ind_{\phi'}(\mu_g^{S,\op})"]  .
\end{tikzcd}
\]
commutes in $\Topv{T''}^{\op}$, which may be checked using the definition of $\mu_g^S$.

The second part of Diagram \eqref{nat3} is
\begin{equation}\label{nat5}
\begin{tikzcd}
\ar[dd,bend right=30mm,"\ind_{\phi'\circ \phi}"'{name=b}]  \Topv{T}^{\op} \ar[rr,"\H^\bullet_T"{name=c}] \ar[d,"\ind_\phi"] && \Sh(T^\C)  \\
 |[alias=X]|\Topv{T'}^{\op} \ar[d,"\ind_{\phi'}"] \ar[rr,"\H^\bullet_{T'}"{name=e}]&&\Sh(T'^\C)  \ar[u,"\phi^*"']  \\
\Topv{T''}^{\op} \ar[rr,"\H^\bullet_{T''}"{name=f}]  && \Sh(T''^\C)\ar[u,"\phi'^*"']  \ar[uu,bend right=30mm,"(\phi'\circ \phi)^*"'{name=a}] \ar[Rightarrow,blueish, from=a, to=u,shorten=2mm,xshift=-1mm,"\sim"']  \ar[Rightarrow,blueish,from=f,to=e,shorten=2mm,"\xi_{\phi'}"'{yshift=0mm},xshift=0mm]  \ar[Rightarrow,blueish,from=e,to=c,shorten=2mm,"\xi_\phi"'{yshift=0mm},xshift=0mm]  \ar[Rightarrow,blueish, from=b, to=X,shorten=2mm,"\sim"] 
\end{tikzcd}
\end{equation}
We need to show that the composite of these natural transformations is equal to 
\[
\xi_{\phi'\circ \phi}: (\phi' \circ \phi)^* \circ \H^\bullet_{T''} \circ \ind_{\phi'\circ \phi} \longrightarrow \H^\bullet_T.
\]
To this end, we consider the diagram
\[
\begin{tikzcd}
\Sh(T^\C) && |[alias=Z]|\Sh(T'^\C) \ar[ll,"\phi^*"'] &&  \ar[ll,"\phi'^*"'] \Sh(T''^\C) \ar[llll,bend right=7mm,"(\phi'\circ \phi)^*"'{name=G}] \\
\\
\Fin{T}^{\op} \ar[uu,color=blueish,"\F^\bullet_T"{name=C}] \ar[rr,color=blueish,"\ind_{\phi}"'xshift=2mm] \ar[rd,hook,"\iota_T"'] && \Fin{T'}^{\op}  \ar[rr,color=blueish,"\ind_{\phi'}"'xshift=2mm] \ar[rd,hook,"\iota_{T'}"'] \ar[uu,color=blueish,"\F^\bullet_{T'}"{name=D}]&& \Fin{T''}^{\op}  \ar[rd,hook,"\iota_{T''}"']  \ar[uu,color=blueish,"\F^\bullet_{T''}"{name=E}]& \\
& \Topv{T}^{\op} \ar[rrrr,bend right=7mm,"\ind_{\phi'\circ \phi}"'{name=H}]\ar[luuu,crossing over,color=red,"\H^\bullet_T"'{name=a}] \ar[rr,color=red,"\ind_{\phi}"']&&|[alias=Y]|\Topv{T'}^{\op}  \ar[luuu,crossing over,color=red,"\H^\bullet_{T'}"'{name=b}] \ar[rr,color=red,"\ind_{\phi'}"']&& \Topv{T''}^{\op}    \ar[luuu,color=red,"\H^\bullet_{T''}"'{name=F}] 
\ar[Rightarrow,from=b,to=a,color=red,shorten=15mm,"\xi_\phi"',xshift=-5mm] \ar[Rightarrow,from=D,to=C,color=blueish,shorten=15mm,"\nu_\phi"',yshift=2mm,"\sim"]
\ar[Rightarrow,from=F,to=b,color=red,shorten=15mm,"\xi_{\phi'}"',xshift=-5mm] \ar[Rightarrow,from=E,to=D,color=blueish,shorten=15mm,"\nu_{\phi'}"',yshift=2mm,"\sim"]
\ar[Rightarrow,from=G,to=Z,"\sim",shorten=2mm]
\ar[Rightarrow,from=H,to=Y,"\sim",shorten=2mm]
\end{tikzcd}
\]
If $\nu_{\phi' \circ \phi}$ is equal to the composite of the natural transformations in all square and semicircular panels (except for those containing $\xi_\phi$ and $\xi_{\phi'}$ respectively), then it follows that the composite in Diagram \eqref{nat5} is equal to $\xi_{\phi'\circ \phi}$, by universal property the right Kan extension. One shows this by taking a finite $T$-complex $X$ and using the definition of $\nu_{\phi}$ to check that the diagram
\[
\begin{tikzcd}
(\phi'\circ \phi)^{-1}\underline{H}^\bullet_{T''}(\ind_{\phi'\circ \phi} X) \otimes_{(\phi'\circ \phi)^{-1}\underline{H}^\bullet(BT'')} (\phi'\circ \phi)^* \O^\bullet_{T''} \ar[r, "\sim"',"\nu_{\phi'\circ \phi}"] \ar[d,"\sim"] & \underline{H}^\bullet_T(X) \otimes_{\underline{H}^\bullet(BT)} \O^\bullet_T \\
\phi^{-1}(\phi'^{-1}\underline{H}^\bullet_{T''}(\ind_{\phi'}(\ind_\phi X))) \otimes_{ \phi^{-1}(\phi'^{-1}\underline{H}^\bullet(BT''))} \phi^*(\phi'^* \O^\bullet_{T''}) \ar[r, "\phi^*(\nu_{\phi'})", "\sim"'] & \phi^{-1}\underline{H}^\bullet_{T'}(\ind_\phi X) \otimes_{\phi^{-1}\underline{H}^\bullet(BT')} \phi^*\O^\bullet_{T'} \ar[u,"\nu_{\phi}","\sim"'] 
\end{tikzcd}
\]
commutes, where the vertical map on the left is induced by the composite of natural isomorphisms $\ind_{\phi'\circ \phi} \to \ind_{\phi'}\circ \ind_\phi$ and $(\phi' \circ \phi)^* \to \phi^*\circ \phi'^*$. It follows that the functor $^S\H^\bullet_\G$ preserves composition.
\end{proof}

\begin{notation}
We will continue to use the notation established in Proposition \ref{functor} in the sequel, and we drop the $S$ from all such notation whenever the subfunctor $S \subset \Hom(\C,-)$ is equal to $\Hom(\C,-)$. We will also use $\H^\bullet_\G$ to denote the associated functor 
\[
\H^\bullet_\G: \Topv{\G}^{\op} \longrightarrow \Gamma(\Orbop,\Sh^{\C^\times}).
\]
\end{notation}

\begin{proposition}\label{natural}
There is a natural transformation
\[
i^*_{P,T,F}:\H^\bullet_\G \longrightarrow \,^S\H^{\bullet}_\G
\]
given by
\[
\begin{tikzcd}
& \ar[phantom] & \\
\Topv{\G}^{\op}  \ar[rrr,bend left=7mm,"ev_{P,T}^{S}" ] \ar[rrr,bend right=7mm,"ev_{P,T}"'] & && \Topv{T}^{\op} \ar[r,"\H_T^\bullet"] &\Sh(T^\C,\O^\bullet)^{\C^\times}\ar[r,hook]& \Sh^{\C^\times}.\\
& \ar[uu,Rightarrow,blueish,shorten=7mm,"i_{P,T,F}^{\op}"'{xshift=1mm},xshift=1mm] & 
\end{tikzcd}
\]
where $i_{P,T,F}: F(P)^{S(T)} \hookrightarrow F(P)^{T^\C} = F(P)$ is the inclusion map. 
\end{proposition}

\begin{proof}
The mapping $F \mapsto i_{P,T,F}$ is clearly natural in $F$. To see that $\H^\bullet_T( i_{P,T,F}^{\op})$ is natural in $(P,T)$, let $\phi: (P,T) \to (P',T')$ be a morphism in $\Orb$ and let $g:P' \to P$ be a morphism in $\G$ representing $\phi$. One easily checks that the diagram
\[
\begin{tikzcd}
\ind_\phi F(P)^{S(T)} \ar[rr,"{i_{P,T,F}}",hook] \ar[d,"\mu^S_g"] &&\ind_\phi  F(P) \ar[d,"\mu_g"] \\
F(P')^{S(T')} \ar[rr,"{i_{P',T',F}}",hook] && F(P')
\end{tikzcd}
\]
commutes in $\Topv{T'}$, which implies that the lower square in the diagram
\[
\begin{tikzcd}
\H^\bullet_{T}(F(P)^{S(T)}) &&\ar[ll,"{i_{P,T,F}^*}"] \H^\bullet_{T}(F(P))  \\
\phi^*\H^\bullet_{T'}(\ind_\phi F(P)^{S(T)}) \ar[u,"\Ran_{\iota_T}(\nu_\phi)"] &&\ar[ll,"{i_{P,T,F}^*}"]\phi^*\H^\bullet_{T'}(\ind_\phi  F(P))\ar[u,"\Ran_{\iota_T}(\nu_\phi)"]  \\
\phi^*\H^\bullet_{T'}(F(P')^{S(T')})  \ar[u,"\phi^*\H^\bullet_{T'}(\mu^S_g)"] &&\ar[ll,"{i_{P',T',F}^*}"] \phi^*\H^\bullet_{T'}(F(P'))\ar[u,"\phi^*\H^\bullet_{T'}(\mu_g)"]
\end{tikzcd}
\]
commutes. The upper square commutes just by naturality of $\Ran_{\iota_T}(\nu_\phi)$, since $i_{P,T,F}$ is a morphism in $\Topv{T}$. The commutativity of both squares together says that the transformation $i_{P,T,F}^*$ is natural in $(P,T)$.
\end{proof}

\subsection{Conditions for $\H^\bullet_\G$ to take values in cocartesian sections}

Let $S$ be a subfunctor of $\Hom(\C,-)$. In this section, we consider the composite functor 
\[
\Topv{\G}^{\op} \xrightarrow{\H^\bullet_\G} \Gamma(\Orbop,\Sh^{\C^\times}) \xrightarrow{\Res_S \circ -} \Gamma(\Orbop,\Sh^{\C^\times}_S).
\]
The goal of this section is to show that there is a full subcategory $\Topv{\G}(S) \subset \Topv{\G}$ whose objects are sent by the composite functor above to cocartesian sections. A cocartesian section 
\[
s: \Orbop \to \Sh_S^{\C^\times}
\]
is a section such that $s(\phi)$ is a cocartesian morphism for each morphism $\phi$ in $\Orbop$.

\begin{definition}
Let $S$ be a subfunctor of $\Hom(\C,-)$. We write $\Topv{\G}(S) \subset \Topv{\G}$ for the full subcategory on functors $F$ satisfying 
\[
F(P) \in \Topv{T}(S(T)) \quad \text{for each} \ (P,T) \in \Orb.
\]
\end{definition}

The main result of this section follows immediately from the following technical result.

\begin{proposition}\label{isom}
Let $F$ be an object in $\Topv{\G}(S)$ and let $\phi:(P,T) \to (P',T')$ be a morphism in $\Orb$. Then $\Res_S(h_\G(\phi))$ is an isomorphism in $\Sh(S,\O^\bullet)^{\C^\times}$, where
\[
h_\G(\phi): \phi^*\H^\bullet_{T'}(F(P')) \to \H^\bullet_T(F(P))
\]
is the morphism in $\Sh(T^\C,\O^\bullet)^{\C^\times}$ defined in the proof of Proposition \ref{functor}.
\end{proposition}

\begin{proof}
Let $g:P' \to P$ be a morphism in $\G$ representing $\phi$. The natural transformation $h_{\G}(\phi)$ was constructed in Proposition \ref{functor} as the composite
\[
\phi^*\H^\bullet_{T'}(F(P')) \to \phi^*\H^\bullet_{T'}(\ind_\phi F(P)) \to \H^\bullet_T(F(P))
\]
where the first map is $(\phi^* \circ \H^\bullet_{T'})(\mu^{\op}_{g,F})$ and the second map is $\xi_{\phi,F(P)}$. We begin by showing that the first map is an isomorphism. By using the adjunction $\ind_\phi \dashv \res_\phi$ we see that the $T'$-map $\mu^{\op}_{g,F}$ factors uniquely as 
\[
\ind_\phi F(P) \cong \ind_\phi \res_\phi F(P') \xrightarrow{\epsilon_{F(P')}} F(P').
\]
where $\epsilon_{F(P')}$ is the counit of the adjunction $\ind_\phi \dashv \res_\phi$ evaluated at $F(P')$. The homeomorphism is induced by the $T$-homeomorphism $F(P) \rightarrow \res_\phi F(P')$ which is, in turn, induced by the homeomorphism $F(g):F(P) \xrightarrow{} F(P')$. It therefore suffices to show that the functor $\phi^* \circ \H^\bullet_{T'}$ sends $\epsilon_{F(P')}$ to an isomorphism. There is a $T'$-equivariant homeomorphism
\[
\ind_\phi \res_\phi F(P') = T' \times_\phi \res_\phi F(P') \cong T'/\phi(T) \times F(P')
\]
sending $[(t,x)] \mapsto (t\phi(T),tx)$, where $T'$ acts diagonally on the right hand side. We may therefore replace $\epsilon_{F(P')}$ with the $T'$-equivariant projection
\begin{equation}\label{secti}
T'/\phi(T) \times F(P') \to F(P'),
\end{equation}
which is sent by $\phi^* \circ \H^\bullet_{T'}$ to a map of $\phi^*\H^\bullet_{T'}(F(P'))$-modules
\begin{equation}\label{allo}
\phi^*\H^\bullet_{T'}(F(P')) \longrightarrow  \phi^*\H^\bullet_{T'}(T'/\phi(T) \times F(P')).
\end{equation}
The map \eqref{allo} cannot be trivial because the map \eqref{secti} has a section. We will show that $\Res_S$ sends \eqref{allo} to an isomorphism by showing that the target of \eqref{allo} is a rank one $\phi^*\H^\bullet_{T'}(F(P'))$-module. It suffices to observe that there are isomorphisms of $\phi^*\H^\bullet_{T'}(F(P'))$-modules
\begin{align*}
\phi^* \H^\bullet_{T'}(F(P') \times T'/\phi(T)) &\cong \phi^*(\H^\bullet_{T'}(F(P')) \otimes_{\O^\bullet_{T'}} \H_{T'}^\bullet(T'/\phi(T)))\\
&\cong \phi^*\H^\bullet_{T'}(F(P')) \otimes_{\phi^*\O^\bullet_{T'}} \phi^*\H_{T'}^\bullet(T'/\phi(T))\\
&\cong \phi^*\H^\bullet_{T'}(F(P')) \otimes_{\O^\bullet_T} \O^\bullet_{T} \\
&\cong \phi^* \H^\bullet_{T'}(F(P'))
\end{align*}
where the third map is induced by the isomorphism in Proposition \ref{orbit}.

It remains to show that $\Res_S(\xi_\phi)$ is an isomorphism at $F(P)$. It will suffice to show that the induced map of stalks at any given $\gamma \in S(T)$ is an isomorphism. Note that, on stalks, the source of $\Res_S(\xi_{\phi,F(P)})$ is isomorphic to
\[
\Res_S(\phi^* \H^\bullet_{T'}(\ind_\phi F(P)))_\gamma \cong \H^\bullet_{T'}(\ind_\phi F(P))_{\phi(\gamma)} \cong \H^\bullet_{T'}((\ind_\phi F(P))^{\phi(\gamma)})_{\phi(\gamma)}
\]
where the second isomorphism is that of Proposition \ref{stalks}. Note also that, for any $\gamma' \in T'^\C$, we have
\[
(\ind_\phi F(P))^{\gamma'} = \begin{cases}  \ind_\phi (F(P)^{\gamma''})  &\text{if  }  \phi(\gamma'') = \gamma' \\ \emptyset &\text{if  } \gamma' \notin \phi(T^\C) \end{cases}.
\]
In particular, our assumption that $F(P)$ is an object in $\Topv{T}(S(T))$ implies that $\ind_\phi F(P)$ is an object in $\Topv{T'}(S(T'))$. Since the $T$-space $F(P)^\gamma$ is finite up to $T$-homotopy, the $T'$-space $\ind_\phi(F(P)^\gamma)$ is also finite up to $T'$-homotopy. It follows that
\[
\Res_S(\xi_{\phi,F(P)})_\gamma: \H^\bullet_{T'}((\ind_\phi (F(P)^{\gamma}))_{\phi(\gamma)} \to \H^\bullet_{T}(F(P)^\gamma)_\gamma
\]
is the maps of stalks induced by $\nu_\phi$, which is an isomorphism by construction. This completes the proof.
\end{proof}

The following corollary follows easily from Proposition \ref{isom} by unpacking the definition of a cocartesian morphism in Definition \ref{opc}.

\begin{corollary}\label{opca}
The composite
\[
\Topv{\G}(S)^{\op} \hookrightarrow \Topv{\G}^{\op} \xrightarrow{\H^\bullet_\G} \Gamma(\Orbop,\Sh^{\C^\times}) \xrightarrow{\Res_S\circ -} \Gamma(\Orbop,\Sh_S^{\C^\times})
\]
takes values in the full subcategory on cocartesian sections in $\Gamma(\Orbop,\Sh_S^{\C^\times})$.
\end{corollary}

\subsection{The definition of $\Ell^\bullet_G$}

We are ready for the main construction of the paper.

\begin{definition}\label{main}
The functor $\Ell^\bullet_{G}$ that we call $G$-equivariant elliptic cohomology is the composite functor
\[
\Fin{G}^{\op} \xrightarrow{} \Topv{\L^2_G}^{\op} \xrightarrow{\H^{\bullet}_{\L^2_G}} \Gamma(\mathrm{Orb}^{\ab,\op}_{\L^2_G},\Sh^{\C^\times}) \xrightarrow{\Res_\Fr \circ -} \Gamma(\mathrm{Orb}^{\ab,\op}_{\L^2_G},\Sh_\Fr^{\C^\times})
\]
where the first functor sends $X$ to $\Map_G(-,X)$. The composite sends $X$ to the section whose value at $(P,T)$ is the triple $(P,T,\Ell^\bullet_{G,P,T}(X))$, where 
\[
\Ell^\bullet_{G,P,T}(X) := \  \H^\bullet_T(\Map_G(P,X))_{\Fr(T)} \ \in  \Sh(\Fr(T),\O^\bullet)^{\C^\times}.
\]
\end{definition}

By functoriality, the sheaf $\Ell^\bullet_{G,P,T}(X)$ is equipped with a $\C^\times \times W(\Aut(P)/T)^{\op}$-action compatible with its $\O^\bullet_{\Fr(T)}$-algebra structure.

\subsection{Main results}

\begin{proposition}\label{locallydokelly}
Let $X$ be a finite $G$-complex and let $(P,T)$ be an object in $\mathrm{Orb}^{\ab}_{\L^2_G}$. The map
\[
\H^\bullet_T(\Map_G(P,X))_{\Fr(T)} \xrightarrow{}  \H^\bullet_T(\Map_G(P,X)^{\Fr(T)})_{\Fr(T)}
\]
induced by the inclusion 
\[
\Map_G(P,X)^{\Fr(T)} \hookrightarrow \Map_G(P,X)
\]
of the subspace of fixed points is an isomorphism natural in $X,P,T$.
\end{proposition}

\begin{proof}
This is the natural transformation of Proposition \ref{natural} with $F = \Map_G(-,X)$ and $S = \Fr$. That it is an isomorphism follows immediately from Proposition \ref{local}.
\end{proof}

The remaining results in this subsection rely on the following fact: if $X$ is a finite $G$-complex and $(P,T)$ an object in $\mathrm{Orb}_{\L^2_G}^{\ab}$, then the $T$-space $\Map_G(P,X)$ is locally finite over $\Fr(T)$ by Proposition \ref{stlk} and Proposition \ref{nbhd}. In other words, the functor $\Ell^\bullet_G$ factors as
\[
\Fin{G}^{\op} \to \Topv{\L^2_G}(\Fr)^{\op} \hookrightarrow \Topv{\L^2_G}^{\op} \xrightarrow{\H^\bullet_{\L^2_G}} \Gamma(\mathrm{Orb}^{\ab,\op}_{\L^2_G},\Sh^{\C^\times}) \xrightarrow{\Res_\Fr} \Gamma(\mathrm{Orb}^{\ab,\op}_{\L^2_G},\Sh_\Fr^{\C^\times}).
\]

\begin{proposition}\label{walkystalky}
Let $X$ be a finite $G$-complex and let $(P,T)$ be an object in $\mathrm{Orb}^{\ab}_{\L^2_G}$. The map induced by the inclusion 
\[
\Map_G(P,X)^{\gamma} \hookrightarrow \Map_G(P,X)
\]
of the subspace of points fixed by $\gamma(\C)$ is an isomorphism
\[
\H^\bullet_T(\Map_G(P,X))_{\gamma} \xrightarrow{\sim}  \H^\bullet_T(\Map_G(P,X)^{\gamma})_{\gamma}
\]
of stalks at $\gamma$ which is natural in $X,P,T$.
\end{proposition}

\begin{proof}
This follows immediately from Proposition \ref{stalks}, since $\Map_G(P,X)$ is locally finite over $\Fr(T)$. It is natural since it is induced by the natural transformation of Proposition \ref{natural}.
\end{proof}

\begin{proposition}\label{opcarry}
The functor 
\[
\Ell^\bullet_G: \Fin{G}^{\op} \to \Gamma(\mathrm{Orb}^{\ab,\op}_{\L^2_G},\Sh^{\C^\times}_{\Fr})
\]
takes values in the full subcategory on cocartesian sections.
\end{proposition}

\begin{proof}
This follows immediately from Corollary \ref{opca}, since $\Map_G(P,X)$ is locally finite over $\Fr(T)$.
\end{proof}

We now consider the reduced version of elliptic cohomology
\[
\widetilde{\Ell}^\bullet_G: \Fin{G}^{\op}_* \longrightarrow \Gamma(\mathrm{Orb}^{\ab,\op}_{\L^2_G},\Sh(\Fr,\O^\bullet\text{-}\mathrm{mod})^{\C^\times})
\]
which is defined for a pointed, finite $G$-complex $X$ and an object $(P,T)$ in $\mathrm{Orb}^{\ab}_{\L^2_G}$ by
\begin{align*}
\widetilde{\Ell}^\bullet_G(X) &:= \ker(\Ell^\bullet_G(X) \to \Ell^\bullet_G(\pt)) \\
&= \ker(\H^\bullet_T(\Map_G(P,X))_{\Fr(T)} \to \H^\bullet_T(\Map_G(P,\pt))_{\Fr(T)}) \\
&= \ker(\H^\bullet_T(\Map_G(P,X))_{\Fr(T)} \to \H^\bullet_T(\pt)_{\Fr(T)}) \\
&= \widetilde{\H}^\bullet_T(\Map_G(P,X))_{\Fr(T)}.
\end{align*}

\begin{proposition}\label{exactness}
The functor $\widetilde{\Ell}^\bullet_G$ is exact on cofiber sequences.
\end{proposition}

\begin{proof}
Let $X \to Y \to Z$ be a cofiber sequence of pointed, finite $G$-complexes and let $(P,T) \in \mathrm{Orb}_{\L^2_G}^{\ab}$. Let $\gamma \in \Fr(T)$ and $p \in P$ and recall the subgroup $\Phi_p(\gamma(\Lambda_\gamma))$ of $G$. The induced sequence 
\[
X^{\Phi_p(\gamma(\Lambda_\gamma))} \to Y^{\Phi_p(\gamma(\Lambda_\gamma))}  \to Z^{\Phi_p(\gamma(\Lambda_\gamma))} 
\]
of fixed-point spaces is also a cofiber sequence and, by Proposition \ref{fixed}, is equivariantly homeomorphic to 
\[
\Map_G(P,X)^\gamma \to \Map_G(P,Y)^\gamma \to \Map_G(P,Z)^\gamma.
\]
Therefore, the latter is a cofiber sequence of finite $T$-complexes, and so the induced sequence
\[
\widetilde{\H}_T^\bullet(\Map_G(P,Z)^\gamma)_\gamma \to \widetilde{\H}_T^\bullet(\Map_G(P,Y)^\gamma)_\gamma \to \widetilde{\H}_T^\bullet(\Map_G(P,X)^\gamma)_\gamma
\]
is exact. Since $\gamma$ was an arbitrary frame, this means that the sequence 
\[
\widetilde{\H}_T^\bullet(\Map_G(P,Z))_{\Fr(T)} \to \widetilde{\H}_T^\bullet(\Map_G(P,Y))_{\Fr(T)} \to \widetilde{\H}_T^\bullet(\Map_G(P,X))_{\Fr(T)}
\]
induced by the cofiber sequence $X \to Y \to Z$ is exact.
\end{proof}

\begin{proposition}\label{susp}
The suspension map of Proposition \ref{suspension} induces a natural isomorphism
\[
\widetilde{\Ell}^{\bullet+1}_G(S^1 \wedge -) \longrightarrow \widetilde{\Ell}^{\bullet}_G(-)
\]
of functors $\Fin{G}^{\op}_* \to \Gamma(\mathrm{Orb}^{\ab,\op}_{\L^2_G},\Sh(\Fr,\O^\bullet\text{-}\mathrm{mod})^{\C^\times})$. 
\end{proposition}

\begin{proof}
Apply the canonical $T$-map 
\[
\Map_G(P,S^1 \wedge X) \cong \Map_G(P,X) \wedge S^1
\]
along with the result of Proposition \ref{suspension}.
\end{proof}

\subsection{Relation to Grojnowski's theory}

In this section we give an explicit calculation in the case that $G$ is a torus of rank $d$ and $P = \bT^2 \times G$ is the canonical trivial bundle over $\Sigma = \bT^2$ (cf. Example \ref{freg}). Let $T$ be the maximal torus
\[
G \times \bT^2 \leq \Map(\bT^2,G) \rtimes \Diff(\bT^2) = \Aut(P),
\]
so that 
\[
\Fr(T) \cong \Fr(\bT^2) \times \Hom(\C,G).
\]
In this case, the Weyl group $W(\Aut(P)/T)$ turns out to be the discrete subgroup 
\[
W(\Aut(P)/T) = \Hom(\bT^2,G) \rtimes \GL_2(\Z)  \leq  \Map(\bT^2,G) \rtimes \Diff(\bT^2),
\]
(for a similar calculation see Proposition 4.6 in \cite{Spong}). The group $\C^\times \times W(\Aut(P)/T)$ acts on $\Fr(T)$ as described in Section \ref{actING}. We describe the action using the coordinates 
\[
\Fr(\bT^2) \times \Hom(\C,G) \cong \X \times \C^d,
\]
where $\X \subset \C^2$ is the subset of pairs $(t_1,t_2)$ such that $\R t_1 + \R t_2 = \C$. The lattice $\Hom(\bT^2,G) \cong \Hom(\bT,G) \times \Hom(\bT,G)$ acts by
\[
(m_1,m_2) \cdot (t_1,t_2,z) = (t_1,t_2,z + t_1m_1 + t_2m_2).
\]
The action of the group $ \GL_2(\Z)$ of linear diffeomorphisms is given by 
\[
\begin{pmatrix} a & b \\ c & d \end{pmatrix} \cdot (t_1,t_2,z) = (at_1+bt_2,ct_1+dt_2,z).
\]
Finally, $\C^\times$ acts on everything by scalar multiplication 
\[
\lambda \cdot(t_1,t_2,z) = (\lambda t_1, \lambda t_2, \lambda z).
\]
This description of the action $\C^\times \times W(\Aut(P)/T)$ on $\Fr(T)$ coincides with a standard description of the $d$-fold universal elliptic curve over the moduli of elliptic curves (see for example page 4 in \cite{Rezk}). More precisely, the induced map of complex analytic stacks 
\[
[\Fr(T) \sslash W(\Aut(P)/T) \times \C^\times] \to [\Fr(\bT^2) \sslash \GL_2(\Z) \times \C^\times] 
\]
presents the $d$-fold universal elliptic curve 
\[
\E^{\times d} \to \M_{ell}(\C)
\]
over the moduli stack $\M_{ell}(\C)$ of complex elliptic curves.

Let $\tau$ be a complex number with positive imaginary part and consider the inclusion 
\[
\iota_\tau: \Hom(\C,G)_\tau \hookrightarrow \X \times \Hom(\C,G)
\]
of the fiber over $(1,\tau) \in \X$. Consider also the quotient  
\[
\zeta_\tau: \Hom(\C,G)_\tau \twoheadrightarrow \Hom(\C,G)_\tau/\Hom(\bT^2,G) \cong \E^{\times d}_\tau
\]
by the action of the lattice $\Hom(\bT^2,G)$ given above, where $\E_{\tau}$ denotes the elliptic curve $\C/\Z + \Z \tau$. We now present a result comparing the main construction in this paper with the construction of Grojnowski in \cite{Groj}. We do not describe the latter construction here (since it is intricate and only tangential to the focus of this paper) but instead refer the interested reader to the original source \cite{Groj}.

\begin{proposition}\label{groj}
With the above hypotheses, the sheaf
\[
\left((\zeta_\tau)_* (\iota_\tau)^* \H^\bullet_T(\Map_G(P,X))_{\Fr(T)} \right)^{\Hom(\bT^2,G)^{\op}}
\]
is naturally isomorphic to Grojnowski's $G$-equivariant elliptic cohomology theory over $\E^{\times d}_{\tau}$ as defined in Section 2.4 in \cite{Groj}.
\end{proposition}

\begin{proof}
There is a natural isomorphism
\begin{align*}
\Ell^\bullet_{G,P,T}(X) := \H^\bullet_T(\Map_G(\bT^2 \times G,X))_{\Fr(T)} \cong \H^\bullet_T(\Map(\bT^2,X))_{\Fr(T)} 
\end{align*}
of sheaves over $\Fr(T) \cong \X \times G^\C$. The right hand side, regarded as a sheaf over $\X \times G^\C$, is then manifestly isomorphic to the sheaf $\G^\bullet(X)$ constructed in Definition 6.2 of \cite{Spong} (where, in that paper, we used $T$ for the group that is here called $G$). It was shown in \cite{Spong} (see Definition 7.8 and Corollary 10.8) that the sheaf $((\zeta_\tau)_* (\iota_\tau)^* \G^\bullet(X))^{\Hom(\bT^2,G)^{\op}}$ is naturally isomorphic to Grojnowski's sheaf over $\E^{\times d}_{\tau}$. This yields the result.
\end{proof}

\subsection{Value of $\Ell^\bullet_{U(1)}$ on a $U(1)$-orbit}

Let $\mu_m \leq U(1)$ be the subgroup of $m$th roots of unity. Set $P = \bT^2 \times U(1) \to \bT^2$ and $T = \bT^2 \times U(1)$. Write $\mathcal{E}[m]$ for the subgroup of points of order $m$ in the universal elliptic curve $\mathcal{E}$. We have
\[
\Ell^\bullet_{P,T,U(1)}(U(1)/\mu_m)^{W(\Aut(P)/T)\times \C^\times} \cong \O^\bullet_{\mathcal{E}[m]}
\]
where the source is the $W(\Aut(P)/T)\times \C^\times$-invariant subsheaf of $\Ell^\bullet_{P,T,U(1)}(U(1)/\mu_m)$. This follows immediately from the calculation in Section 9 in \cite{Spong}.

\subsection{Proof that $\Map_G(P,X)$ is locally finite over $\Fr(T)$}\label{locally}

Recall from Lemma \ref{inject} the homomorphism
\[
\Phi_p: \Gau(P) \to G
\]
sending $f$ to the unique element $g \in G$ such that $f(p) = p\cdot g$.

\begin{lemma}\label{diff}
Let $H$ be a closed subgroup of $G$, let $(P,T)$ be an object in $\mathrm{Orb}_{\L^2_G}^{\ab}$, choose a point $p$ in $P$, choose a frame $\gamma \in \Fr(T)$ and suppose that $(t_1,t_2)$ is a choice of basis for the lattice $\Lambda_\gamma$. Then $\gamma$ has a $\C^\times$-equivariant neighbourhood $U$ in $\Fr(T)$ such that for each $\gamma' \in U$ and for each $i = 1,2$ we have
\[
\Phi_p(\gamma'(t_i')) \in H \quad \implies \quad  \Phi_p(\gamma'(st_i')\gamma(st_i)^{-1}) \in H \quad \forall s \in \R
\]
where $t_1',t_2' \in \Lambda_{\gamma'}$ is the basis corresponding to $t_1,t_2 \in \Lambda_\gamma$ via the diffeomorphism
\[
\tilde{\eta}_{\gamma',\gamma}: \C/\Lambda_{\gamma'} \xrightarrow{} \C/\Lambda_{\gamma}.
\]
(Explicitly, this means that $t_i' := \eta_{ \gamma',\gamma}^{-1}(t_i)$.)
\end{lemma}

\begin{proof}
We write $A_p$ for $A_p(T)$ and $H_p$ for the intersection $A_p \cap H$. Then $H_p \leq A_p$ is a compact abelian subgroup since $H$ is closed and $A_p$ is compact abelian. Let $A^0_p$ and $H^0_p$ denote the respective identity components and consider the diagram 
\[
\begin{tikzcd}
\Hom(\R/\Z,H^0_p) \ar[r,hook] \ar[d,hook] & \Hom(\R/\Z,A^0_p) \ar[r,two heads] \ar[d,hook] & \Hom(\R/\Z,A^0_p/H^0_p) \ar[d,hook] \\
\Hom(\R,H^0_p) \ar[r,hook] \ar[d,two heads,"\ev_1"] & \Hom(\R,A^0_p) \ar[r,two heads] \ar[d,two heads,"\ev_1"] & \Hom(\R,A^0_p/H^0_p) \ar[d,two heads,"\ev_1"]  \\
H^0_p \ar[r,hook] & A^0_p \ar[r,two heads] & A^0_p/H^0_p  
\end{tikzcd}
\]
with exact rows and columns induced by the exact sequences $0 \to \Z \to \R \to \R/\Z \to 0$ and $1 \to H^0_p \to A^0_p \to A^0_p/H^0_p \to 1$. The identifications labeled by $\ev_1$ are given by evaluation at $1 \in \R$. By inspecting the diagram one sees that the preimage of $H^0_p$ in $\Hom(\R,A^0_p)$ under $\ev_1$ is precisely the subgroup of elements of the form $\alpha + \beta$, where $\alpha:\R \to A^0_p$ has image contained in $H^0_p$ and $\beta:\R \to A^0_p$ is the pullback of some homomorphism $\R/\Z \to A^0_p$. In other words
\begin{align*}
\ev_1^{-1}(H^0_p) &= \Hom(\R,H^0_p) + \Hom(\R/\Z,A^0_p) \quad  \subset \Hom(\R,A^0_p) \\
&= \coprod_{[\beta] \in \Hom(\R/\Z,A^0_p/H^0_p)} \Hom(\R,H^0_p) + \beta 
\end{align*}
Since the image of $\Hom(\R/\Z,A^0_p/H^0_p)$ in $\Hom(\R,A^0_p/H^0_p)$ is a lattice, there is an open neighbourhood $V$ of the trivial map in $\Hom(\R,A^0_p)$ which intersects only the component of $\ev_1^{-1}(H^0_p)$ indexed by $[\beta] = [0]$. Therefore, if $\phi$ is a one-parameter subgroup in $V$ sending $1$ into $H^0_p$, then the image of $\phi$ is contained in $H_p$. 

The following argument applies for $i = 1,2$. Write $g_i$ for the element $\Phi_p(\gamma(t_i))$ in $A_p$. Since $H_p$ is compact, it consists of finitely many connected components in $A_p$. It follows that there exists an open neighbourhood $V_i$ of $g_i \in A_p$ that does not intersect any component of $H_p$ that does not contain $g_i$. Let $W_i$ be the preimage of $V_i$ under the continuous map
\begin{align*}
\Fr(T) &\rightarrow A_p \\
\gamma' &\mapsto \Phi_p(\gamma'(t_i')).
\end{align*}
This map is invariant under the $\C^\times$-action on $\Fr(T)$, since it sends $\lambda \cdot \gamma'$ to
\[
\Phi_p ((\lambda \cdot \gamma') (t_i'/\bar{\lambda} ) ) =  \Phi_p ( \gamma' (\bar{\lambda} t_i'/\bar{\lambda})) =  \Phi_p(\gamma'(t_i')),
\]
where we have used that $\eta_{\lambda \gamma',\gamma}^{-1}(t_i) = \frac{1}{\bar{\lambda}} t_i'$. Therefore $W_i$ is a $\C^\times$-equivariant open neighbourhood of $\gamma$. Furthermore, let $U_i$ be the preimage of $V$ under the continuous map
\begin{align*}
\Fr(T) &\to \Hom(\R,A_p)  \\
\gamma' &\mapsto \Phi_p(\gamma'(st_i')\gamma(st_i)^{-1})
\end{align*}
which is also $\C^\times$-invariant, for the same reason as the previous map. Thus $V_i$ is also a $\C^\times$-equivariant open neighbourhood of $\gamma$. Then
\[
U := U_1 \cap U_2 \cap W_1 \cap W_2
\]
is a $\C^\times$-equivariant open neighbourhood of $\gamma$ and we will show that it has the property claimed in the lemma. Let $\gamma'$ be a frame in $U \subset \Fr(T)$ and suppose that 
\[
g_i' := \Phi_p(\gamma'(t_i')) \in H_p = A_p \cap H.
\]
This immediately implies that $g_i'$ is contained in $V_i \subset A_p$, since if $\gamma'$ lies in $U$ then it also lies in $W_i$. Since $g_i' \in H_p$, both $g_i'$ and $g_i$ lie in the same component of $H_p$ since $V_i$ was chosen with this property. This means that $g_i'g_i^{-1}$ lies in $H^0_p$. Therefore, since $\gamma' \in U_i$, the one-parameter subgroup 
\[
s \mapsto \Phi_p(\gamma'(st_i')\gamma(st_i)^{-1})  \in A_p
\] 
lies in $V \subset \Hom(\R,A_p)$, which implies by the discussion above that its image is contained in $H_p \leq H$ for all $s \in\R$. This completes the proof.
\end{proof}

\begin{proposition}\label{nbhd}
Let $X$ be a finite $G$-complex, let $(P,T)$ be an object in $\mathrm{Orb}_{\L^2_G}^{\ab}$ and let $\gamma \in \Fr(T)$. There exists a $\C^\times$-equivariant open neighbourhood $U$ of $\gamma$ such that 
\[
\Map_G(P,X)^{U} := \bigcup_{\gamma' \in U} \Map_G(P,X)^{\gamma'} = \Map_G(P,X)^{\gamma}. 
\]
\end{proposition}

\begin{proof}
Fix a point $p \in P$ and suppose that $t_1,t_2$ are a basis for $\Lambda_\gamma$. If $\gamma'$ is another frame in $\Fr(T)$, then we continue to write $t_1',t_2'$ for the basis of $\Lambda_{\gamma'}$ corresponding to $t_1,t_2$ via $\eta_{\gamma',\gamma}$. By Lemma \ref{diff}, for each isotropy group $H$ in $X$ there exists a $\C^\times$-equivariant open neighbourhood $U_H$ of $\gamma \in \Fr(T)$ such that
\[
\Phi_p(\gamma'(t_i')) \in H \quad \implies \quad \Phi_p(\gamma'(st_i')\gamma(st_i)^{-1}) \in H \quad \forall s \in \R
\]
for all $\gamma' \in U_H$. Since $X$ is a finite $G$-complex, it has only has finitely many isotropy groups. It follows that the intersection 
\[
U := \bigcap_{H\, \in \, \text{isotropy}(X)} \ U_H
\]
indexed over all isotropy groups of $X$ is a $\C^\times$-equivariant open neighbourhood of $\gamma$ with the property that
\[
\Phi_p(\gamma'(t_i')) \in H \quad \implies \quad \Phi_p(\gamma'(st_i')\gamma(st_i)^{-1}) \in H \quad \forall s \in \R
\]
for all $\gamma' \in U$ and any isotropy group $H$ of $X$. To prove the claim, it will suffice to show that 
\[
\Map_G(P,X)^{\gamma'} \subset \Map_G(P,X)^{\gamma}
\]
for all $\gamma' \in U$. Let $f \in \Map_G(P,X)^{\gamma'}$ and let $H \leq G$ be the isotropy group of $f(p) \in X$. Firstly, note that we have
\[
f(p) \cdot \Phi_p(\gamma'(t_i')) = f(p \cdot \Phi_p(\gamma'(t_i'))) = f(\gamma'(t_i')(p)) = f(p)
\]
since $f$ is both $G$-equivariant and fixed by $\gamma'$. Therefore, 
\[
\Phi_p(\gamma'(t_i')) \in H \quad \text{for} \ i = 1,2.
\]
Since $\gamma' \in U$ and $U$ has the property stated above, this implies that 
\[
\Phi_p(\gamma'(st_i')\gamma(st_i)^{-1}) \in H \qquad \text{for all} \ s \in \R \ \text{and} \ i = 1,2.
\]
It follows that
\begin{align*}
f(\gamma(st_i)(p)) &= f((\gamma(st_i) \gamma'(st_i')^{-1} \gamma'(st_i'))(p)) \\
& = f(\gamma'(st'_i)(p)) \cdot \Phi_p(\gamma(st_i)\gamma'(st_i')^{-1}) \\
& = f(p) \cdot \Phi_p(\gamma'(st_i')\gamma(st_i)^{-1})^{-1} \\
& = f(p) 
\end{align*}
for all $s\in \R$ and $i = 1,2$. Since $t_1,t_2$ is an $\R$-basis for $\C$, this means that $f(\gamma(z)(p)) = f(p)$ for all $z \in \C$. Therefore $f\in \Map_G(P,X)^{\gamma}$.
\end{proof}

\begin{proposition}\label{fixed}
Let $X$ be a finite $G$-complex, let $(P,T)$ be an object in $\mathrm{Orb}_{\L^2_G}^{\ab}$, let $\gamma \in \Fr(T)$ and choose $p \in P$. Then the map 
\[
\ev_p: \Map_G(P,X)^\gamma =  \Map_G(P,X)^{\gamma(\C)} \to X^{\Phi_p(\gamma(\Lambda_\gamma))}
\]
given by evaluation at $p$ is a homeomorphism. There are isomorphisms of Lie groups
\[
T/\widebar{\gamma(\C)} \xleftarrow{\sim} \Gau(T)/\widebar{\gamma(\Lambda_\gamma)} \xrightarrow{\sim} A_p(T)/\overline{\Phi_p(\gamma(\Lambda_\gamma))}
\]
where the map on the left is induced by the inclusion of $\Gau(T) := T\cap \Gau(P)$ into $T$ and the map on the right is induced by the isomomorphism 
\[
\Phi_p: \Gau(T) \xrightarrow{\sim} A_p(T).
\]
The map $\ev_p$ is equivariant with respect to the composite 
\[
T/\widebar{\gamma(\C)} \cong A_p(T)/\overline{\Phi_p(\gamma(\Lambda_\gamma))}
\]
of these isomorphisms. Finally, the image of any map in $\Map_G(P,X)^\gamma$ is contained in a single $G$-orbit of $X$. 
\end{proposition}

\begin{proof}
It is clear that $\ev_p$ is continuous, since $\Map_G(P,X)^\gamma$ has the subspace topology induced by $\Map_G(P,X)$, which has the compact-open topology. For a given $x \in X^{\Phi_p(\gamma(\Lambda_\gamma))}$, we construct a $G$-map $f:P \to X$ such that $f(p) = x$. For an arbitrary point $p'$ in $P$, define
\[
f(p') = x \cdot h
\] 
where $h$ is an element in $G$ such that there exists a complex number $z$ with $\gamma(z)(p)\cdot h = p'$. To see that $f(p')$ does not depend on the choice of $h$, suppose that $h' \in G$ and $z' \in \C$ also satisfy $\gamma(z')(p)\cdot h' = p'$. Then $z - z'$ lies in $\Lambda_\gamma$, and we have
\[
p' = \gamma(z)(p)\cdot h = \gamma(z + z' - z') (p) \cdot h = \gamma(z') (p) \cdot \Phi_p(\gamma(z-z'))h.
\] 
Therefore $h' =  \Phi_p(\gamma(z-z'))h$, which implies that $x \cdot h = x \cdot h'$, since $x$ is fixed by $\Phi_p(\gamma(z-z')) \in \Phi_p(\gamma(\Lambda_\gamma))$. So $f$ is well defined. Furthermore, $f$ is $G$-equivariant, since 
\[
f(p'\cdot g) = x \cdot hg = f(p') \cdot g
\]
by definition of $f$. We now show that $f$ is fixed by $\gamma(\C)$. Let $z'' \in \C$ be arbitrary and let $p'' = \gamma(z'')(p')$. Then we have $\gamma(z''+z)(p)\cdot h = \gamma(z'')(p') = p''$ and so
\[
f(p'') = x \cdot h = f(p').
\]
Therefore, we have defined a two-sided inverse to $\ev_p$, so that $\ev_p$ is bijective. To see that the inverse is continuous, let $V_x$ be the set of maps in $\Map_G(P,X)^\gamma$ corresponding to an open neighbourhood $U_x$ of $x \in X^{\Phi_p(\gamma(\Lambda_\gamma))}$ under the established bijection. Then $V_x$ is precisely the set of $G$-maps $f: P \to X$ sending $p$ into $U_x$ such that $f$ is fixed by $\gamma$. In other words, $V_x$ is the intersection of $\Map_G(P,X)^\gamma$ with the set of maps in $\Map_G(P,X)$ sending $p$ into $U_x$, and is therefore open in the subspace topology on $\Map_G(P,X)^\gamma \subset \Map_G(P,X)$.

It is clear that the homomorphism of Lie groups induced by $\Phi_p$ is an isomorphism. The map
\[
\Gau(T)/\widebar{\gamma(\Lambda_\gamma)} \xrightarrow{\sim} T/\widebar{\gamma(\C)} 
\]
induced by the inclusion $\Gau(T) \hookrightarrow T$ is an isomorphism because any given coset $j \widebar{\gamma(\C)}$ in $T/\widebar{\gamma(\C)}$ has a representative $j  \gamma(z)^{-1}$ in $\Gau(T)$, where $z$ is any complex number such that $\tilde{\gamma}([z]) = \tilde{\pi}(j)$. The $T/\widebar{\gamma(\C)}$-action on $\Map_G(P,X)^\gamma$ clearly factors through this isomorphism. Thus, to show equivariance of $\ev_p$ with respect to 
\[
T/\widebar{\gamma(\C)} \cong \Gau(T)/\widebar{\gamma(\Lambda_\gamma)} \cong A_p(T)/\overline{\Phi_p(\gamma(\Lambda_\gamma))}
\]
it suffices to observe that   
\[
f(j(p)) = f(p) \cdot \Phi_p(j)
\]
for each $j \in \Gau(T)$, just by definition of $\Phi_p$. Finally, by definition of the inverse map of $\ev_p$, it is clear that any $f \in \Map_G(P,X)^\gamma$ takes values in the orbit $f(p) \cdot G \subset X$. Therefore, the final claim of the proposition is also proved.
\end{proof}

\begin{corollary}\label{stlk}
Let $X$ be a finite $G$-complex. For any object $(P,T)$ in $\mathrm{Orb}_{\L^2_G}^{\ab}$ and any frame $\gamma \in \Fr(T)$, the $T$-space $\Map_G(P,X)^\gamma$ is $T$-homotopy equivalent to a finite $T$-complex.
\end{corollary}

\begin{proof}
We write $A_p$ for $A_p(T)$ and $\Phi_p$ for $\overline{\Phi_p(\gamma(\Lambda_\gamma))}$. Since $X$ is a finite $G$-complex, there exists an $A_p$-complex $\check{X}$ and an $A_p$-homotopy equivalence $X \to \check{X}$ such that $\check{X}^{\Phi_p}$ is a finite $A_p$-complex, by Theorem A(1) in \cite{Illman}. Therefore, $X^{\Phi_p}$ is $A_p/\Phi_p$-homotopy equivalent to a finite $A_p/\Phi_p$-complex.
It now follows from Proposition \ref{fixed} that $\Map_G(P,X)^\gamma$ is $T/\widebar{\gamma(\C)}$-homotopy equivalent to a finite $T/\widebar{\gamma(\C)}$-complex. Pulling the action back along $T \to T/\widebar{\gamma(\C)}$ yields the result.
\end{proof}


\begin{thebibliography}{9}

\bibitem{BET1} Berwick-Evans, D. and Tripathy, A., 2018. A model for complex analytic equivariant elliptic cohomology from quantum field theory. {\em arXiv preprint arXiv:1805.04146.}

\bibitem{BET2} Berwick-Evans, D. and Tripathy, A., 2021. A de Rham model for complex analytic equivariant elliptic cohomology. {\em Advances in Mathematics}, 380, p.107575.

\bibitem{Devoto} Devoto, J.A., 1996. Equivariant elliptic homology and finite groups. {\em Michigan Mathematical Journal,} 43(1), pp.3-32.

\bibitem{GepnerMeier} Gepner, D. and Meier, L., 2023. On equivariant topological modular forms. {\em Compositio Mathematica}, 159(12), pp.2638-2693.

\bibitem{GKV} Ginzburg, V., Kapranov, M. and Vasserot, E., 1995. Elliptic algebras and equivariant elliptic cohomology. arXiv preprint q-alg/9505012.

\bibitem{Groj} Grojnowski, I., 2007. Delocalised equivariant elliptic cohomology. In {\em Elliptic cohomology, volume 342 of London Mathematical Society Lecture Note Series,} pp. 114–121. Cambridge Univ. Press, Cambridge, 2007.

\bibitem{Illman} Illman, S., 1990. Restricting the transformation group in equivariant CW complexes. {\em Osaka Journal of Mathematics,} 27, pp. 191-206.

\bibitem{Kitchloo} Kitchloo, N., 2019. Quantization of the modular functor and equivariant elliptic cohomology. In {\em Homotopy theory: tools and applications, volume 729 of Contemp. Math.,}  pp/ 157–177. Amer. Math. Soc., Providence, RI.

\bibitem{LRS} Landweber, P.S., Ravenel, D.C. and Stong, R.E., 1995. Periodic cohomology theories defined by elliptic curves. {\em Contemporary Mathematics,} 181, pp.317-317.

\bibitem{Lurie} Lurie, J., 2009, May. A survey of elliptic cohomology. In {\em Algebraic Topology: The Abel Symposium 2007} (pp. 219-277). Berlin, Heidelberg: Springer Berlin Heidelberg.

\bibitem{Nara} Narasimhan, M.S. and Ramadas, T.R., 1979. Geometry of SU (2) gauge fields. {\em Communications in Mathematical Physics,} 67, pp.121-136.

\bibitem{Rezk} Rezk, C., 2019. Looijenga line bundles in complex analytic elliptic cohomology. {\em Tunisian Journal of Mathematics,} 2(1), pp.1-42.

\bibitem{Rezktalk} Rezk, C., 2022. {\em A double-loop group picture for equivariant elliptic cohomology.} Available from: \href{https://www.youtube.com/watch?v=uthA93IyyUk}{https://www.youtube.com/watch?v=uthA93IyyUk} [Accessed 25th May 2024]

\bibitem{Rosu1} Rosu, I., 2001. Equivariant elliptic cohomology and rigidity. {\em American Journal of Mathematics,} 123(4), pp.647-677.

\bibitem{Rosu} Rosu, I., 2003. Equivariant K-theory and equivariant cohomology. {\em Mathematische Zeitschrift,} 243, pp.423-448.

\bibitem{Segal} Segal, G., 1987. Elliptic cohomology. {\em Séminaire Bourbaki,} 88, pp.161-162.

\bibitem{Spong} Spong, M., 2021. A construction of complex analytic elliptic cohomology from double free loop spaces. {\em Compositio Mathematica}, 157(8), pp. 1853–1897.

\bibitem{StolzTeichner} Stolz, S. and Teichner, P., 2011. Supersymmetric field theories and generalized cohomology. {\em Mathematical foundations of quantum field theory and perturbative string theory,} 83, pp.279-340.

\bibitem{tomDieck} Dieck, T.T., 1987. {\em Transformation groups.} de Gruyter.

\bibitem{Witten} Witten, E., 1987. Elliptic genera and quantum field theory. {\em Communications in Mathematical Physics,} 109(4), pp.525-536.


\end{thebibliography}
\end{document}